\documentclass[11pt, letterpaper, reqno]{amsart}
\usepackage[utf8]{inputenc}
\usepackage{amsmath}
\usepackage{amsfonts}
\usepackage{amssymb}
\usepackage{amsthm}
\usepackage{dsfont}

\usepackage[backref=page,breaklinks]{hyperref}
\usepackage{breakurl}
\usepackage{stmaryrd}
\usepackage{mathdots}
\usepackage{tikz-cd}
\usepackage{float}
\usepackage[linesnumbered,lined,boxed,commentsnumbered,ruled,algosection]{algorithm2e}
\usepackage[shortlabels, inline]{enumitem}
\usepackage{colonequals}
\usepackage{cleveref}
\usepackage{mathrsfs}
\usepackage{placeins}
\usepackage{colonequals}
\usepackage{multicol}

\DeclareMathOperator{\disc}{disc}

\DeclareMathOperator{\Gal}{Gal}
\DeclareMathOperator{\Aut}{Aut}
\DeclareMathOperator{\End}{End}

\DeclareMathOperator{\trd}{trd}
\DeclareMathOperator{\Mat}{Mat}
\DeclareMathOperator{\Jac}{Jac}
\DeclareMathOperator{\Prym}{Prym}

\DeclareMathOperator{\M}{M}
\DeclareMathOperator{\GL}{GL}
\DeclareMathOperator{\SL}{SL}
\DeclareMathOperator{\PSL}{PSL}
\DeclareMathOperator{\GSp}{GSp}
\DeclareMathOperator{\SU}{SU}

\DeclareMathOperator{\KS}{KS}
\DeclareMathOperator{\id}{id}

\DeclareMathOperator{\Cor}{Cor}
\DeclareMathOperator{\Res}{Res}

\DeclareMathOperator{\coker}{coker}

\DeclareMathOperator{\Sh}{Sh}
\DeclareMathOperator{\MT}{MT}
\DeclareMathOperator{\CC}{\mathbb{C}}
\DeclareMathOperator{\C}{\mathbb{C}}
\DeclareMathOperator{\QQ}{\mathbb{Q}}
\DeclareMathOperator{\Q}{\mathbb{Q}}

\DeclareMathOperator{\R}{\mathbb{R}}
\DeclareMathOperator{\ZZ}{\mathbb{Z}}
\DeclareMathOperator{\Z}{\mathbb{Z}}

\DeclareMathOperator{\Ac}{\mathcal{A}}
\DeclareMathOperator{\Mod}{\mathcal{M}}

\newcommand{\calO}{\mathcal{O}}
\newcommand{\calX}{\mathcal{X}}

\newcommand{\PP}{\mathbb{P}}
\newcommand{\Magma}{\textsf{Magma}}
\renewcommand{\H}{\mathcal{H}}
\newcommand{\wt}{\widetilde}

\theoremstyle{definition}
\newtheorem{definition}[algocf]{Definition}

\newtheorem{rem}[algocf]{Remark}

\theoremstyle{plain}
\newtheorem{thm}[algocf]{Theorem}
\newtheorem*{thm*}{Theorem}
\newtheorem{cor}[algocf]{Corollary}
\newtheorem{lem}[algocf]{Lemma}
\newtheorem{prop}[algocf]{Proposition}
\newtheorem{conj}[algocf]{Conjecture}
\numberwithin{equation}{section}

\title[Mumford-type Shimura curves]{Mumford-type Shimura curves contained in the Torelli locus}
\author[Bouchet, Hanselman, Pieper, Schiavone]{Thomas Bouchet, Jeroen Hanselman, Andreas Pieper,\\ and Sam Schiavone}
\address{
  Thomas Bouchet,
  Universit\'e C\^ote d'Azur,
}
\email{thomas.bouchet@univ-cotedazur.fr}
\address{
  Jeroen Hanselman,
  RPTU Kaiserslautern-Landau
}
\email{hanselman@mathematik.uni-kl.de}
\address{
  Andreas Pieper,
  Universit\"at Duisburg-Essen
}
\email{andreas.pieper@uni-due.de}
\address{
  Sam Schiavone,
}
\email{sam.schiavone@gmail.com}

\date{\today}

\begin{document}

\begin{abstract}
    In this article we give explicit equations for two Shimura families of genus $4$ curves whose Jacobians are abelian fourfolds of Mumford type. These are the first explicit examples of abelian varieties over $\Q$  with geometric endomorphism algebra $\mathbb{Z}$ and additional Hodge tensors.
\end{abstract}

\maketitle
\tableofcontents
\newpage
\section{Introduction}
\subsection{Main results}

The main goal of this article is to present two explicit families of abelian varieties of the type constructed by Mumford in \cite{MumfordShimura}. In particular, we give explicit equations for two families of genus $4$ curves parametrized by Shimura curves, such that the Jacobians of these curves are of Mumford type. This realizes the longstanding goal of providing explicit examples of Mumford's construction, as well as answering several questions of Moonen and Oort \cite{MoonenOort}.

Let $\Ac_g$ be the coarse moduli space of principally polarized complex abelian varieties of dimension $g$, $\Mod_g$ the moduli space of curves and $j:\Mod_g \to \Ac_g$ the Torelli map. The image $j(\Mod_g) = T_g^o$ is called the open Torelli locus and its Zariski closure $T_g$ is called the Torelli locus. A subvariety $S$ of $\Ac_g$ is called a special subvariety if $S$ is an irreducible component of a Shimura variety.

Coleman conjectured \cite{Coleman}, as an analogue to the Manin-Mumford conjecture (which has been proven by Raynaud \cite{Raynaud}), that for fixed $g\geqslant 4$ there are only finitely many CM points in the open Torelli locus. The Andr\'e-Oort conjecture, which has been proven conditionally on GRH by work of Klingler, Yafaev and Ullmo \cite{AndreOort1} \cite{AndreOort2} and unconditionally by Tsimerman \cite{Tsimerman}, says that if $S \subset\Ac_g$ is an irreducible subvariety, then $S$ is special if and only if the CM points in $S$ are Zariski dense. Combining this with Coleman's conjecture we get the following.

\begin{conj}[Coleman--Oort] For large enough $g$ there are no special subvarieties $S \subset \Ac_g$ of positive dimension that are contained in the Torelli locus $T_g$ and that meet $T_g^o$.
\end{conj}

For $g \leq  7$ it is known that special subvarieties of positive dimension in the open Torelli locus exist. For $g < 4$ this is trivially true as $\dim T_g^o = \dim \Ac_g$. Special subvarieties with $g = 4$ and $g=6$ have been found by de Jong and Noot \cite{DeJong1991}---these examples go back to the work of Shimura \cite{Shimura}---and examples of genus 5 and genus 7 have been found by Rohde \cite{Rohde}. For a summary of results, see \cite{MoonenOort}. All previously known examples of such families arise from Galois covers of curves \cite{FGP, FPP} and are of PEL type. Moonen computed in \cite{Moonen_subvars} a list of all possible families that can arise as cyclic covers of $\mathbb{P}^1$.

A natural question to ask is if there are any other kinds of families. Up until now, no Hodge-type special loci with geometric generic endomorphism algebra $\mathbb{Z}$ were known \cite[Question 6.3]{MoonenOort}. In \cite{MumfordShimura}, Mumford gave a construction of  compact Shimura curves not of PEL type parametrizing a family of abelian varieties whose geometric endomorphism algebra is indeed generically isomorphic to $\Z$. These abelian varieties of Mumford type have been studied from a variety of perspectives. Various researchers have studied their Galois representations \cite{Noot-GaloisRepns}, those with CM \cite{NootMumford}, the representations of their Mumford--Tate groups \cite{Galluzzi-Corestriction}, and the K3 surfaces associated to them via the Kuga--Satake construction \cite{Galluzzi-KugaSatake, Zhu-K3}.

However, in the more than 50 years since the publication of Mumford's article, no explicit examples of abelian varieties of Mumford type have been produced, despite promising strategies towards that goal \cite{Calegari}. 

In \cite[Problem 1]{OpenProblems}, Gross asked if there are abelian fourfolds of Mumford type that are Jacobians of curves. Baldi--Klinger--Ullmo~\cite[Theorem 3.17, Remark 3.18]{BKU} proved that there exist infinitely many such examples over $\overline{\mathbb{Q}}$, while noting that providing equations of such curves remained an open problem.
In this article, we present two explicit families over $\mathbb{Q}$ of genus $4$ curves whose Jacobians are abelian fourfolds of Mumford type. These give us two new special families in $T_4$, thereby giving an explicit answer to Gross's question, as well as several questions of Moonen and Oort \cite[Questions 6.2, 6.3, 6.5]{MoonenOort} (Question 6.3 was proposed earlier by Oort in 1995 published in \cite[Question 7B]{Oort-OpenProblemsArithmetic}).

Let $\H$ be the complex upper half-plane and for $p,q,r \in \Z_{\geq 2}$ let $\Delta(p,q,r)$ denote the triangle group with these parameters. (See \Cref{subsec:triangle} below for a precise definition.)
\begin{thm}[\Cref{prop:Shimura}, \Cref{thm:main}] \label{thm:main_intro}
${}$
\begin{enumerate}
    \item[i)]
Consider the family of hyperelliptic curves
    \begin{equation} \label{eq:237}
       \begin{aligned}
          C_{7,t}: y^2 &= t\left(\left(t - \frac{27}{16}\right) x^{10} - \frac{567}{64} x^{9} - \frac{189}{4} t x^{8} + \left(-84 t^{2} - \frac{189}{4} t\right)\right. x^{7}\\
          &\left.\quad - 189 t^{2} x^{6} - \frac{189}{2} t^{2} x^{5} + 84 t^{3} x^{4} + 108 t^{3} x^{3} - 28 t^{4} x\right)\,.
       \end{aligned}
    \end{equation}
     Then $C_{7,t}$ is a Shimura family of Mumford type. The analytification of the associated Shimura curve $\calX_7$ is $\Delta(2,3,7) \backslash \H$. 
    \item[ii)] Consider the canonically embedded genus $4$ curves in $\PP^3_{X, Y, Z, W}$ given by
    \begin{equation} \label{eq:239}
        \begin{aligned}
        C_{9,t}: 0 &= X Z-Y^{2}\\
        0 &= 3W^3 + t(t - 1)\Big( ( 5X^2 + 6XY + 2tYZ + 3tZ^2 )3W \\
        &+ (-2t + 9)X^3 + 22tX^2Y + 21tX^2Z + (-14t^2 + 18t)XYZ \\
        &+ t^2XZ^2 + 6t^2YZ^2 + (-3t^3 + 6t^2)Z^3 \Big) \, .
        \end{aligned}
    \end{equation}
    Then $C_{9,t}$ is a Shimura family of Mumford type. The analytification of the associated Shimura curve $\calX_9$ is $\Delta(2,3,9) \backslash \H$.
    \end{enumerate}
\end{thm}
We found these families while performing a large-scale computation of CM abelian fourfolds contained in the Torelli locus (whose findings will be published elsewhere). We noticed that, after excluding the fibers belonging to the de Jong--Noot families~\cite{DeJong1991}, all but two of the remaining curves appeared to fall into two families. In order to find the equations of these families we first applied the method of~\cite{HPS} for recovering the equation of a genus four curve from a small period matrix. We then interpolated the invariants~\cite{BouchetInvs} of these curves, and, using the reconstruction process of~\cite{BouchetRecon}, obtained explicit equations for the families.
To prove that these families are indeed parametrized by Shimura curves, we proceed as follows. We first introduce the Deligne–Mumford stacks (\Cref{def:stacks}) that serve as the putative moduli stacks of the families. We then extend a criterion of Viehweg–Zuo~\cite[Proposition 0.3]{VZ} to the setting of stacky curves. We verify that the families in \Cref{eq:237,eq:239} satisfy the hypotheses of this criterion by computing their semistable reduction (\Cref{prop:Shimura}). Finally, relying on Moonen’s classification of Shimura curves in $\mathcal{A}_g$~\cite{MoonenCrvs}, we show that our Shimura families are of Mumford type (\Cref{thm:main}).

\subsection{Organization of the article}

We begin \Cref{sec:background} by recalling the definitions of Mumford-type abelian varieties and arithmetic triangle groups. In \Cref{sec:reconstruction} we explain how we found \Cref{eq:237,eq:239} using a combination of our work \cite{BouchetInvs, BouchetRecon, HPS}. In \Cref{sec:mumford}, we then prove the main result showing that the two explicit families indeed give fourfolds of Mumford type. We conclude in \Cref{sec:future} with directions for further research, including connections with K3 surfaces and conjectures relating these families to hypergeometric motives.

\subsection{Code} Some calculations in this paper were verified using the \Magma{} computer algebra system \cite{Magma}. We have additionally written code that gives direct access to the two explicit families. Our code depends on the genus 4 invariants package by Bouchet \cite{BouchetInvs-Github}. All of the code used in this article can be found in the GitHub repository \cite{Github-Mumford}.

Any step performed using \Magma{} will be clearly indicated as such, and the corresponding scripts in the repository will be referenced. We used \Magma{} v2.28-7 for our computations.

\subsection{Acknowledgments}
The authors would like to thank Shiva Chidambaram, Edgar Costa, Andrea Gallese, Nazim Khelifa, Davide Lombardo, Frans Oort, Bjorn Poonen, Jeroen Sijsling, Andrew Sutherland, and Yuwei Zhu for helpful conversations during the writing of this article.

The second named author is supported by MaRDI, funded by the Deutsche Forschungsgemeinschaft (DFG), project number 460135501, NFDI 29/1. The last named author was supported by the Simons Collaboration in Arithmetic Geometry, Number Theory, and Computation via Simons Foundation grant 550033. 

\section{Background}
\label{sec:background}
\subsection{Abelian varieties of Mumford type}
\label{subsec:mumford_defn}

Let $A = \C^g/\Lambda$ be a complex abelian variety of dimension $g$. Letting $V_{\R} \colonequals \Lambda\otimes \R$, the isomorphism $\C^g\to  V_{\R}$ endows $V_{\R}$ with a complex structure, i.e., a matrix $J \in \GL_{\R}(2g, V_{\R})$ such that $J^2 = -I.$ The eigenvalues of the matrix $J$ in turn induce a decomposition of $V_{\R} \cong H_1(A, \Z)\otimes_{\Z}\C$ into two eigenspaces $H^{-1,0}\oplus H^{0,-1}$, equipping $V$ with a $\Z$-Hodge structure of type $\{(-1,0),(0,-1)\}$. A Hodge structure of type $\{(-1,0),(0,-1)\}$ on such a $\Z$-module $\Lambda$ can equivalently be described by a choice of a map \[h:\CC^*\to \GL_{\R}(\Lambda\otimes \R), \hspace{1em} h(a+bi)\mapsto aI +bJ.\]

By tensoring $H_1(A, \ZZ)$ with $\Q$ we switch to considering isogeny classes of abelian varieties instead of isomorphism classes of abelian varieties. We can define a $\Q$-Hodge structure on a real vector space $V_{\R} := V_{\Q}\otimes_{\Q}\R$ for a $\Q$-module $V_{\Q}$ using the map $h$.
One can interpret Hermitian symmetric domains as moduli varieties for polarized Hodge structures. In \cite{MumfordShimura}, Mumford used this to define Shimura curves classifying polarized abelian varieties with a certain Hodge structure; he referred to these families as being of Hodge type. The most commonly known Shimura curves are of PEL type, i.e., they classify abelian varieties equipped with polarization, endomorphism ring, and level structure. But in the same paper Mumford gave a construction of a Shimura curve parametrizing a family of abelian varieties that is of Hodge type, but not of PEL type. These abelian varieties were characterized by a different property, namely that their Mumford--Tate groups are ``small'' in a certain sense.

\begin{definition}
Let $A$ be a complex abelian variety and let $h:\CC^*\to \GL(V_{\R})$ be the homomorphism associated to the $\Q$-Hodge structure on $V_{\Q} = H_1(A, \QQ)$. The \emph{Mumford--Tate group} $\MT(A)$ of $A$ is defined to be the smallest algebraic group $G$ defined over $\QQ$ such that $G_{\R}$ contains the image of $h$ in $\GL(V_{\R})$. 
\end{definition}

Having a ``small'' Mumford--Tate group has some interesting consequences. Suppose $A$ is defined over a number field $F$. Let $\rho_\ell$ be the Galois representation given by letting $\Gal(\overline{F}/F)$ act on the $\ell$-adic Tate-module $T_\ell$. The Mumford--Tate group in some sense ``controls'' what the image $\rho_\ell(\Gal(\overline{F}/F))$ looks like: more precisely, the Mumford--Tate conjecture predicts that
\begin{equation*}
\overline{\rho_\ell\left(\Gal(\overline{F}/F)\right)}^0 \cong {\MT(A)\times_{\QQ} \QQ_\ell} \, ,
\end{equation*}
although the Mumford-Tate conjecture is not known in the case of general Mumford-type abelian fourfolds; see \cite[Section 2.4.7]{Moonen-MTC}.
What is interesting about Mumford's family of abelian varieties is that they generically have endomorphism ring isomorphic to $\ZZ$ despite having a ``small'' Mumford--Tate group. This contrasts with Serre's open image theorem \cite{Serre} which states that, for every abelian variety $A$ with $\End(A) \cong \Z$ and whose dimension is either 2, 6 or odd, $\rho_\ell$ is surjective for all sufficiently large $\ell$.

 We now recall Mumford's construction.
    Let $K$ be a totally real cubic field with archimedean places $v_1, v_2, v_3$, and let $B$ be a $K$-quaternion algebra. Write $\sigma_i: K\to \overline{\QQ},\, i=1,2,3$ for the embeddings of $K$ into $\overline{\QQ}$ and let $B_i = B \otimes_{K, \sigma_i} \overline{\QQ}$. The Galois group $\Gal(\overline{\QQ}/\QQ)$ naturally permutes the $B_i$. This induces a natural semilinear action on $D = B_1\otimes_{\overline{\QQ}} B_2 \otimes_{\overline{\QQ}} B_3$. The corestriction $\Cor_{K/\mathbb{Q}}(B)$ is defined to be the algebra $D^{\Gal(\overline{\QQ}/\QQ)}$ consisting of all elements fixed by this Galois action. 

    \begin{samepage}
    For Mumford's construction, we further assume that 
        \begin{enumerate}
            \item $B$ splits at $v_1$ and is non-split at the other two archimedean places; and
            \item $\Cor_{K/\mathbb{Q}}(B) \cong \Mat_{8,8}(\Q)$. 
        \end{enumerate} 
    \end{samepage}
    Consider the algebraic group $G'$ over $\Q$
    $$G'(\mathbb{Q})=\{x\in B^{\times} \mid x x^\ast =1 \}$$
    and its central extension $G$ with $\mathbb{G}_{m,\Q}$:
    $$
    G(\mathbb{Q}) = \mathbb{Q}^\times \cdot\{x\in B^{\times} \mid x x^\ast =1 \} \subset B^\times.
    $$
    Define the norm map $N: G(\mathbb{Q}) \to \Cor_{K/\mathbb{Q}}(B)$ by $\gamma \mapsto \sigma_1(\gamma)\otimes \sigma_2(\gamma)\otimes \sigma_3(\gamma)$. Since $\Cor_{K/\mathbb{Q}}(B) \cong \Mat_{8,8}(\mathbb{Q})$, this gives us a representation
    $$
    \rho: G \longrightarrow \GL_{8, \QQ}
    $$
    defined on a vector space $V\cong \QQ^8$. 
    Note that $G$ is a $\QQ$-form of the $\mathbb{R}$-algebraic group \[(\mathbb{G}_{m,\mathbb{R}}\times \SL(2, \R)\times \SU(2)\times \SU(2))/(-1,-I,-I,-I)\,.\]
 One can show that this induces a unique alternating form preserved up to scaling by the image of $\rho$ on $V$; i.e., the image of $\rho$ lands inside $\GSp(8, \QQ)$. Let $\Lambda$ be any lattice in $V$. Using $\rho$ and the alternating pairing we can define a polarized Hodge structure of weight $1$ on $\Lambda$ by taking $\CC^*\to \GL(\Lambda_{\R})$ to be the composition $\rho_{\R}\circ h$ where $h$ is defined to be
\[h: \mathbb{C}^\times \to (\R^\times \times \SL(2, \R)\times \SU(2)\times \SU(2))/(-1,-I,-I,-I) \cong G(\mathbb{R)}\]
given by
\[
h(r \cdot e^{i\theta}) = \left(r, \begin{pmatrix}
\cos(\theta) & \sin(\theta)\\
-\sin(\theta) & \cos(\theta)
\end{pmatrix}, I, I \right).
\]
Now $(G, [h])$ is a Shimura datum~\cite[Section 6]{Deligne}. This has the following consequence: letting $H$ be the
connected component of the centralizer of $h$ in $G(\R)$ we get $G(\R)/H$, a disjoint union of bounded symmetric domains (in our case upper halfplanes) whose points parametrize polarized Hodge structures. 
Every arithmetic subgroup $\Gamma \subset G'$ such that $\Gamma(\Lambda) = \Lambda$ gives us a family of polarized abelian varieties over $Y \cong \Gamma  \backslash G(\R)/H$. (If $\Gamma$ is not neat this quotient should be considered as a stack.) By construction, members of these families will generically have Mumford--Tate group isomorphic to $G$.

Note that $Y$ could be disconnected. We will call both $Y$ and the connected components of $Y$ Shimura curves.

\subsection{CM abelian fourfolds of Mumford type} \label{subsec:CM_Mumford}

We now recall some results about the CM points of Shimura curves of Mumford type and the corresponding CM abelian fourfolds.

Let $K$ be a totally real cubic number field and $B$ be a quaternion algebra over $K$ satisfying the assumptions of \Cref{subsec:mumford_defn}. Following \cite[\S3.4]{NootMumford}, let $L$ be a maximal subfield of $B$, and assume that $L$ is a totally imaginary extension of $K$, and also that there does not exist an imaginary quadratic field $F$ such that $L = F K$. Let $\wt{K}$ and $\wt{L}$ be the Galois closures of $K$ and $L$, respectively, and let $H \colonequals \Gal(\wt{K}/\Q)$ and $G \colonequals \Gal(\wt{L}/\Q)$.

Since $L$ is a quadratic extension of a cubic field, there is an injective homomorphism $\varphi: G \rightarrow \Z/2 \wr S_3$. Note that $\Z/2 \wr S_3$ is isomorphic to $\Z/2\times S_4$, which is the automorphism group of a cube. The assumption that $L$ does not contain an imaginary quadratic field implies that $\varphi$ is an isomorphism if $K/\Q$ is not Galois; see \cite[5.1.2]{Dodson}. Otherwise, $K/\Q$ is a Galois $\Z/3$ extension and the image of $\varphi$ is $\Z/2\wr \Z/3\cong \Z/2\times A_4$. This occurs in our examples, as $K=\Q(\zeta_7)^+$ (resp., $\Q(\zeta_9)^+$).

In both the Galois and non-Galois case there is a transitive action of $G$ on the vertices of a cube. We get a subfield $E\subset \wt{L}$ such that the embeddings $E\hookrightarrow \C$ are in bijection with the vertices of the cube. Noot \cite[\S3.6]{NootMumford} shows that taking the four embedding corresponding to a face of the cube defines a CM type $\Phi$ of $E$. Furthermore, he proves that there is a CM abelian variety $A$ of type $(E, \Phi)$ corresponding to a point of the Mumford type Shimura curve associated to $B$. Conversely, every simple Mumford type CM abelian variety arises in this way. (Non-simple CM abelian varieties are the result of allowing $L$ to contain an imaginary quadratic field $F$, in which case $\End^0(A)\cong L\oplus F$.)

\subsection{Triangle groups}
\label{subsec:triangle}

The Shimura curves $\calX_7$ and $\calX_9$ that parametrize the families given in \Cref{thm:main_intro} happen to arise as quotients of the upper half-plane by the action of triangle groups. Let $p,q,r \in \Z_{\geq 2} \cup \{\infty\}$ with $p \leq q \leq r$. The triple $(p,q,r)$ is \textit{spherical}, \textit{Euclidean}, or \textit{hyperbolic} according to whether the value
$$
 1 - \frac{1}{p} - \frac{1}{q} - \frac{1}{r}  
$$
is respectively negative, zero, or positive. These cases correspond to the triangle $T = T(p,q,r)$ with angles $\pi/p$, $\pi/q$, and $\pi/r$ naturally living on the sphere, the Euclidean plane, or two-dimensional hyperbolic space, respectively. Denote by $H$ this natural ambient space of the triangle $T$.

We consider the group of orientation-preserving isometries of $H$ generated by rotations about the vertices of $T$. Let $P, Q$, and $R$ be the vertices of $T$ with angles $\pi/p$, $\pi/q$, and $\pi/r$, respectively, and let $\delta_p, \delta_q$, and $\delta_r$ be the rotations about $P,Q$, and $R$ by $2\pi/p$, $2\pi/q$, and $2\pi/r$. Letting $\Delta = \Delta(p,q,r)$ be the group generated by the rotations $\delta_p, \delta_q, \delta_r$, one can show that $\Delta$ has the presentation
\begin{equation*}
    \Delta = \langle \delta_p, \delta_q, \delta_r \mid \delta_p^p = \delta_q^q = \delta_r^r = \delta_r \delta_q \delta_p =1 \rangle \, .
\end{equation*}
Tesselating $H$ by the translates of the triangle $T$, one can view $\Delta$ as the group of orientation-preserving isometries of $H$ that preserve the tesselation. For more details on triangle groups, see \cite[Chapter II]{Magnus}. 

\begin{figure}[htbp]
    \centering
    \begin{minipage}[b]{0.45\textwidth}
        \centering
        \includegraphics[scale=0.35]{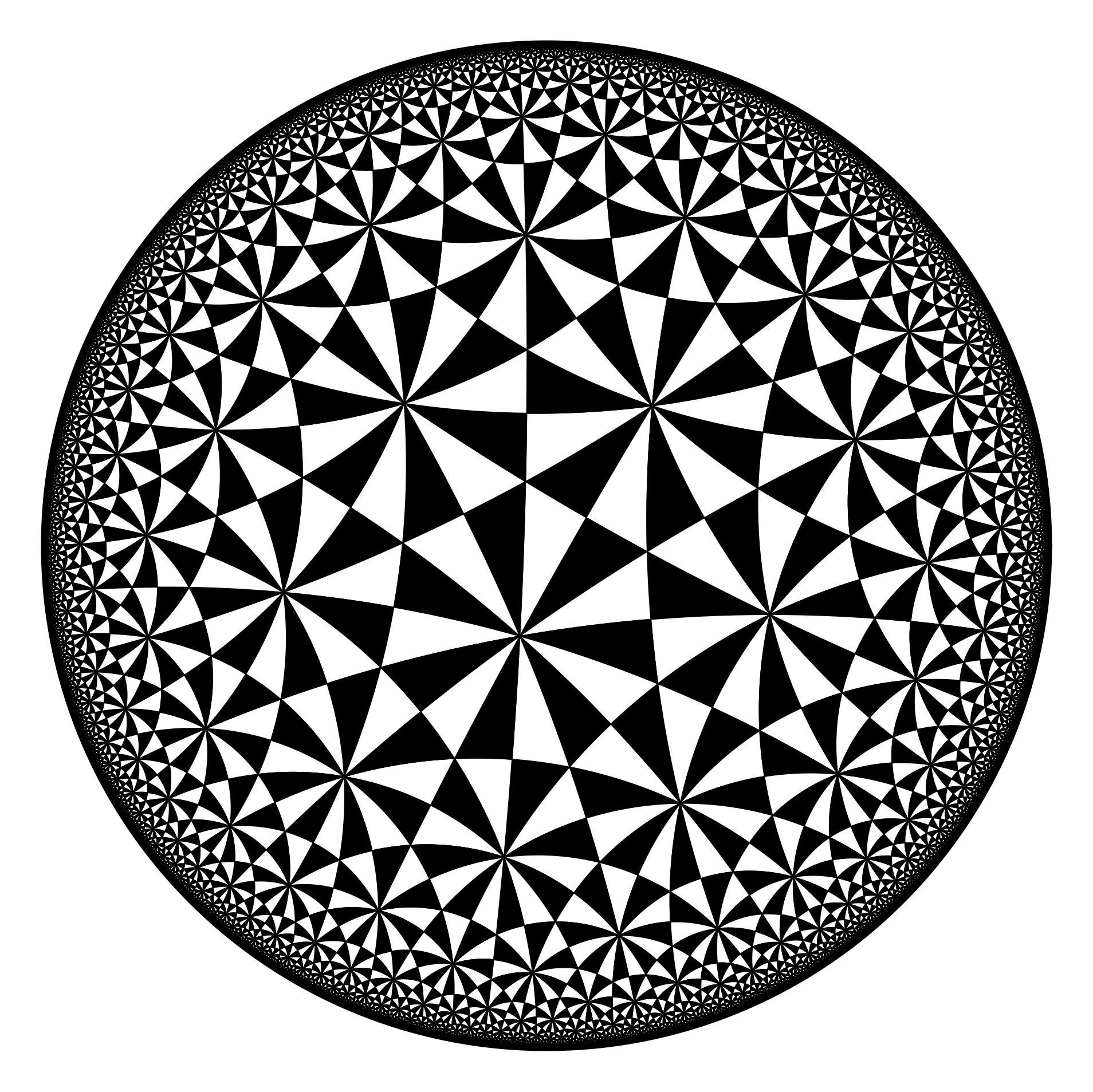}
        \label{fig:tess1}
    \end{minipage}
    \hfill
    \begin{minipage}[b]{0.45\textwidth}
        \centering
        \includegraphics[scale=0.35]{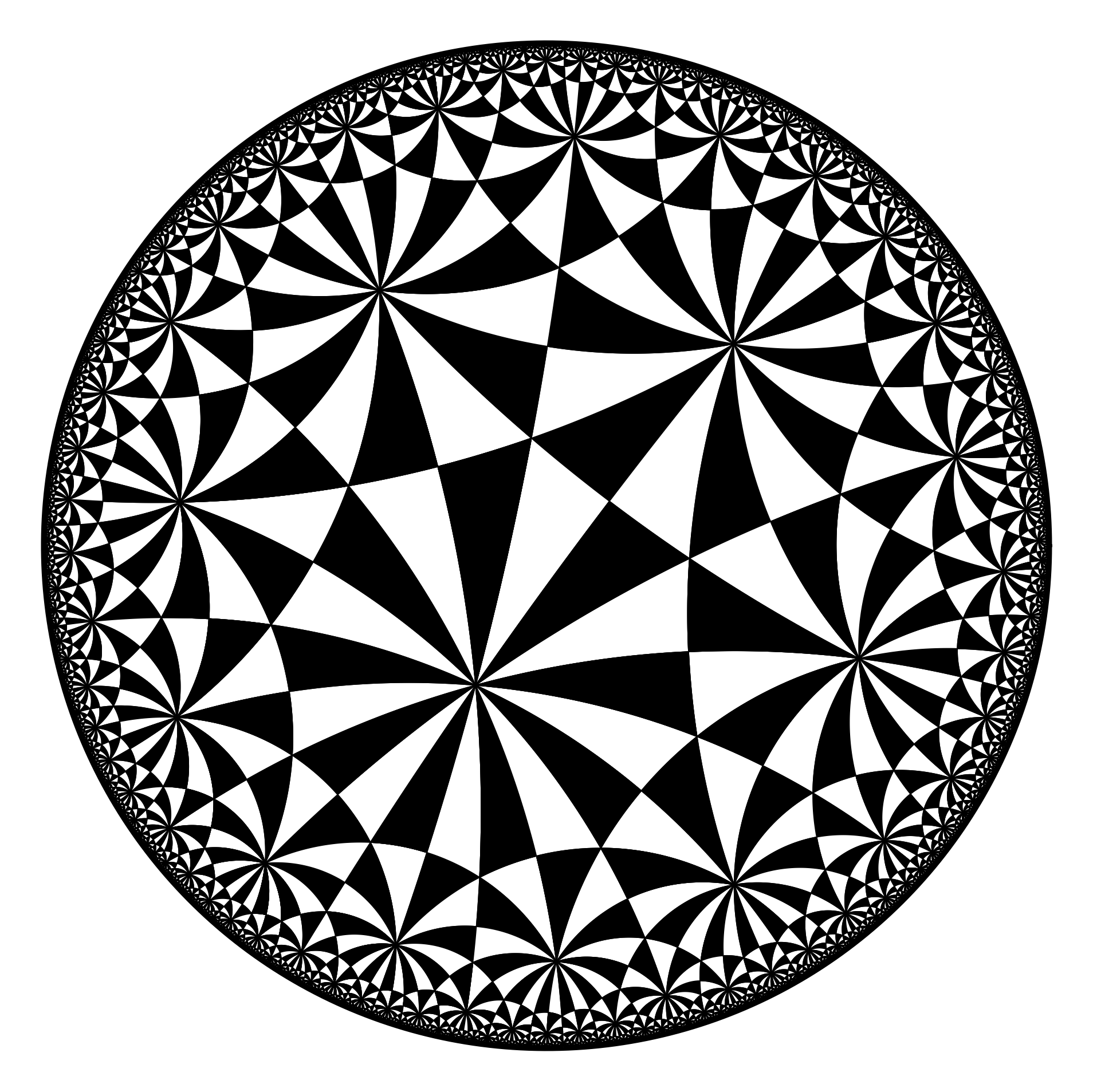}
        \label{fig:tess2}
    \end{minipage}
    \caption{Tesselations with $(p,q,r)=(2,3,7)$, $(2,3,9)$}
    \label{fig:two_tess}
\end{figure}

\begin{rem}
Triangle groups arise as the monodromy groups of the ${}_2 F_1$ hypergeometric differential equation.
In \cite{DeligneMostow}, Deligne and Mostow generalized these groups to higher dimensions, studying discrete subgroups of the group of isometries of the complex hyperbolic ball and their connection to the monodromy of higher dimensional hypergeometric functions. Interestingly, several of these ball quotients give rise to special subvarieties of the Torelli locus. For example, all the special families arising from cyclic covers of $\mathbb{P}^1$ have this property \cite{Moonen_subvars}. 

We hope to explore the connection between the families of curves given in \Cref{thm:main_intro} and hypergeometric functions in future work. For some initial steps in this direction, see \Cref{subsec:hypgeom}.
\end{rem}

A special feature of the triangle groups $\Delta(2,3,7)$ and $\Delta(2,3,9)$ is the fact that they are arithmetic. Let $B$ be a quaternion algebra over a totally real number field $K$ that is split at exactly one real place. Then
\begin{equation*}
    B \hookrightarrow B \otimes_{K} \R \cong \M_2(\R) \times \mathbb{H}^{n-1}
\end{equation*}
where $\mathbb{H}$ denotes the Hamilton quaternions. Let $\iota_\infty$ denote the projection onto the first factor. Given an order $\calO \subseteq B$, let $\calO^1$ be the group of units of $\calO$ of reduced norm $1$. Then $\iota_\infty(\calO^1) \subseteq \SL_2(\R)$. Let $\Gamma(B,\calO)$ be the projectivization of $\iota_\infty(\calO^1)$, i.e., its image under the quotient map $\SL_2(\R) \to \PSL_2(\R)$. A triangle group is \textit{arithmetic} if it is commensurable to $\Gamma(B,\calO)$ for some quaternion algebra $B$ and some order $\calO \subseteq B$ as above.

Takeuchi proved that there are exactly 85 triples $(p,q,r)$ with $p \leq q \leq r$ such that $\Delta(p,q,r)$ is arithmetic \cite{TakeuchiArith}, and further organized them into commensurability classes \cite{TakeuchiCommens}. In the case of $\Delta(2,3,7)$ (resp., $\Delta(2,3,9)$), we have $K = \Q(\zeta_7)^+$ (resp., $K = \Q(\zeta_9)^+$).

By global class field theory, given a subset of places of $K$ with even cardinality there is a unique quaternion algebra over $K$ ramified at exactly those places. Denoting the infinite (real) places of $K$ by $v_1, v_2, v_3$, let $B$ be the quaternion algebra ramified at exactly $v_2$ and $v_3$. Note that such a quaternion algebra satisfies the conditions in Mumford's construction, as its corestriction has trivial Brauer class \cite[Theorem 11]{Riehm}. One can show that the class set of $B$ is trivial, so $B$ has a unique maximal order $\calO$ up to conjugation. As shown in \cite[Table (3)]{TakeuchiCommens}, then $\Gamma(B,\calO)$ is isomorphic to $\Delta(2,3,7)$ (resp., $\Delta(2,3,9)$).

To see this explicitly, consider $K = \Q(\zeta_n)^+,\, n\in \{7,9\}$ with generator $\nu = \zeta_n + \zeta_n^{-1}$. 
Define the quaternion algebra $B_n$ by $i^2=-1, j^2=\nu^2-3$. Then the elements
\begin{align*}
\delta_p \colonequals i, \qquad \delta_q \colonequals \frac{1}{2}+\frac{\nu}{2}i + \frac{1}{2}j, \qquad \delta_r \colonequals \delta_p^{-1} \delta_q^{-1}
\end{align*}
generate a subgroup isomorphic to $\Delta(2,3,7)$ (resp., $\Delta(2,3,9)$).

\begin{rem}\label{rem:2311}
Note that one can also set $n=11$ in these formulas, showing that the $\Delta(2,3,11)$ triangle group is arithmetic (see also \cite{TakeuchiArith}). The center of $B$ is the field $K = \Q(\zeta_{11})^+$ which has degree $5$. Nevertheless, there is an analogue of Mumford's construction (see \cite[\S5]{MoonenCrvs}) giving rise to a Shimura curve $\mathcal{X}_{11}$ in $\mathcal{A}_{16}$. It seems that $\calX_{11}$ is not contained in the Torelli locus \cite[Theorem A]{LuZuo}. (See also \Cref{rem:LuZuoError}.) However, we believe that there is an interesting family of curves which has the universal family of $\calX_{11}$ as a factor of its Jacobian.
\par Moreover, all the families $\mathcal{X}_7$, $\mathcal{X}_9$, and $\mathcal{X}_{11}$ are all related to K3 surfaces (see Section \ref{subsec:K3}), and thus it would also be worthwhile to study $\calX_{11}$ in that context.
\end{rem}

\section{Discovery of the special families $C_{7,t},\, C_{9,t}$} \label{sec:reconstruction}

We discovered the families given in \Cref{eq:237,eq:239} while looking for Jacobians among the principally polarized fourfolds with CM by the fields in the L-functions and modular forms database (LMFDB) \cite{lmfdb}. 
This computation (whose details and results will be published in future work) unexpectedly produced a large number of hyperelliptic examples. All but two of the examples arose from CM fields with Galois closure containing $\Q(\zeta_7)^+$, suggesting that they belonged to a single family parametrized by a Shimura curve. A similar phenomenon appeared in some non-hyperelliptic examples, for which the 
Galois closure of the CM fields contained $\Q(\zeta_9)^+$. In both situations, the CM fields had Galois group $\Z/2 \times A_4$, evoking the results of Noot~\cite{NootMumford} on CM abelian varieties of Mumford type recalled in \Cref{subsec:CM_Mumford}. Taken together, these observations led us to further conjecture that the Shimura curves in question were of Mumford type.

With a sufficient number of curves from these two hypothesized families in hand, we attempted to recover equations for each family by interpolating their invariants~\cite{BouchetInvs}. A rough outline of our strategy is as follows.

\begin{enumerate}
    \item Compute many approximate equations of CM curves that are expected to lie in one of the families, following~\cite{HPS}.
    \item For each of these curves, determine their complex invariants using~\cite{BouchetInvs}. Next, identify algebraic relations between the invariants that heuristically hold for all of the curves, with the ultimate goal of expressing each invariant as a rational function in terms of the minimal number of invariants---in our case, a single invariant. These expressions have coefficients in $\Q$.
    \item Apply the reconstruction formulas from~\cite{BouchetRecon} to derive an exact expression of the family, consistent with the computed invariants.
\end{enumerate}

More precisely, for certain degree 8 CM fields (listed in \Cref{sec:tables}) with Galois group $\Z/2\times A_4$ and Galois closure containing $\Q(\zeta_7)^+$ or $\Q(\zeta_9)^+$, we applied the method of~\cite{DIS-CM, CM-g3}, implemented in \Magma{}~\cite{DIS-Github}, working with a precision of $10{,}000$ decimal digits. This produced the small period matrices of some principally polarized abelian varieties with CM by the maximal order of one of these fields. Observe that all the discriminants of these CM fields are fourth powers. This is not true for all degree 8 CM fields with Galois group $\Z/2\times A_4$ and Galois closure containing $\Q(\zeta_7)^+$ or $\Q(\zeta_9)^+$; the pattern seems to be a consequence of the maximality of the order.

Using an implementation of~\cite{fast-theta} in FLINT \cite{flint}, we evaluated their theta constants to the same precision, and the vanishing behavior of the Schottky--Igusa modular form \cite{Igusa} indicated that all these principally polarized abelian varieties were Jacobians. We were then able to reconstruct approximate equations of the corresponding curves, as described in~\cite[Theorem 3.11]{HPS}.

A classical strategy for computing examples of curves with CM is as follows. Starting from an approximate equation of a CM curve with sufficiently high precision, one can heuristically recover an exact equation by using invariants that classify the curve up to geometric isomorphism~\cite{CM-g2,CM-g3}. Once properly normalized, these invariants lie in the field of moduli of the curve; by Shimura’s reciprocity law~\cite{Deligne}, this field is a number field, so the invariants are algebraic numbers. Identifying them then reduces to applying the LLL algorithm, provided the initial precision is sufficiently high.

Once the exact invariants are determined, the next step is reconstruction: given a set of invariants, find a curve realizing them. This problem was solved by Mestre~\cite{Mestre} for generic hyperelliptic curves and, more recently, for generic hypersurfaces by the first author~\cite{BouchetRecon}.
Our aim was slightly different: rather than reconstructing individual CM curves, we sought to determine equations for the entire families to which we believed these curves belong. Assuming that such families can be described by a small set of parameters, these parameters must appear in the invariants, thereby increasing the number of relations they satisfy.
By studying these relations, one can hope to retrieve such parameters. 

\begin{sloppypar}
In both cases, we chose a distinguished invariant, which we denote by $t$ (different for each family), in terms of which every other invariant can be expressed rationally (and, in the case of the $C_{7,t}$ family, even polynomially), suggesting that the CM curves belong to families parametrized by the projective line $\mathbb{P}^1$. These expressions can be found in the file {\texttt{Shimura/Shimura\_models.m}} in~\cite{Mumford-Github}.
\end{sloppypar}

All that remained was the reconstruction of the equations. After a precomputation step where we expressed the coefficients of the polynomials appearing in the reconstruction formulas~\cite[Cor. 6.1.7 and Prop. 6.3.1]{BouchetRecon}, in terms of the invariant $t$, we were able to derive expressions of $C_{7,t}$ and $C_{9,t}$ of high degrees in $t$. After several minimizing steps, following~\cite{Elsenhans-Stoll}, we obtained \Cref{eq:237,eq:239}.
The process of constructing these equations is described in the file {\texttt{Shimura/construction.m}} in~\cite{Mumford-Github}.

\section{Proofs of the main theorems}\label{sec:mumford}
In this section we prove that the curves given by \Cref{eq:237,eq:239} are Shimura families of Mumford type. We do so by first defining Deligne--Mumford stacks that encapsulate the three elliptic points of $\Delta(2,3,7)$ (resp., $\Delta(2,3,9)$). Then we descend the Jacobian of the families to these stacks and apply a criterion for universal families over Shimura curves due to Viehweg--Zuo \cite[Proposition 0.3]{VZ}. To show that the Shimura families are of Mumford type, we compute the fundamental group of the stack and use the classification of Shimura curves in $\mathcal{A}_g$ due to Moonen \cite{MoonenCrvs}.

\subsection{Construction of $\mathcal{X}_7$ and $\mathcal{X}_9$ as DM-stacks 
}
We begin by introducing the Deligne--Mumford stacks. Since it does not create additional complications, we will treat the general case of triangle groups $\Delta(p,q,r)$. We are thus looking for a smooth DM-stack $\mathcal{X}(p,q,r)$ over $\mathbb{Q}$ with the following properties:
\begin{itemize}
    \item the coarse space is $\mathbb{P}^1_\mathbb{Q}$;
    \item the generic inertia group is $\mathbb{Z}/2$;
    \item it has exactly three elliptic points; and
    \item at the elliptic point of order $p$, (resp., $q$, $r$) the inertia group jumps to $\mathbb{Z}/2p$ (resp., $\mathbb{Z}/2q$, $\mathbb{Z}/2r$).
\end{itemize}
\par We now give our definition of $\mathcal{X}(p,q,r)$.
\begin{definition}\label{def:stacks}
    Let $p,q,r$ be positive integers such that $r$ is odd and $\gcd(p,q)=1$. We define $\mathcal{X}(p,q,r)$ as the $r^{\text{th}}$ root stack of the weighted projective space $\mathbb{P}(2q, 2p)$ at the point $1$.
\end{definition}
By construction $\mathcal{X}(p,q,r)$ is a proper smooth DM-stack with coarse space $\mathbb{P}^1$. We will denote by $t: \mathcal{X}(p,q,r)\rightarrow \mathbb{P}^1$
the map to the coarse space obtained from $\mathbb{P}(2q, 2p)\rightarrow \mathbb{P}^1,\,(x:y)\mapsto \frac{x^p}{y^q}$. The map $t$ is analogous to the map $j: \overline{\mathcal{M}}_{1,1} \rightarrow \mathbb{P}^1$ which sends a moduli point to its $j$-invariant.
\par Consider the open subset $U \colonequals \mathbb{P}^1\setminus\{0,1,\infty \}$ and $\mathfrak{U} \colonequals t^{-1}(U) $.
\begin{lem}\label{lem:UBZ2}
One has $\mathfrak{U} \cong U \times B\mathbb{Z}/2$, where $BG$ denotes the classifying stack.
\end{lem}
\begin{proof}
    Since we have $\mathcal{X}(p,q,r) \setminus \{t^{-1}(1)\} = \mathbb{P}(2q, 2p) \setminus \{1\}$, it suffices to show that $\mathbb{P}(2q, 2p) \setminus \{0, 1, \infty\} \cong \mathbb{P}^1\setminus\{0,1,\infty \} \times B\mathbb{Z}/2$. Indeed, 
    consider the map $\mathbb{G}_m \rightarrow \mathbb{P}(2q, 2p)$ given by $s\mapsto (s^a: s^b)$ where $a,b \in \mathbb{Z}$ are chosen to satisfy $ap-bq=1$.

    We want to compute the fiber product $\mathbb{G}_m\times_{\mathbb{P}(2q, 2p)} \mathbb{G}_m$; note that weighted projective space is a quotient by $\mathbb{G}_m$. Thus the fiber product is a closed subscheme of $\mathbb{G}_m^3$. Choosing coordinates $s_1, s_2, \lambda$ on $\mathbb{G}_m^3$ the equations are given by
    \begin{equation}
        \begin{aligned}
        \mathbb{G}_m \times_{\mathbb{P}(2q, 2p)} \mathbb{G}_m&=\{ s_1^a \lambda^{2q}=s_2^a,\, s_1^b \lambda^{2p}=s_2^b\}\\
        &=\{ s_1=s_2,\, \lambda=\pm 1\}\\
        &\cong \mathbb{G}_m \times \mathbb{Z}/2.
        \end{aligned}
    \end{equation}
    This gives an explicit isomorphism $\mathbb{G}_m \times B\mathbb{Z}/2\cong {\mathbb{P}(2q, 2p) \setminus \{0, \infty \}}$.
\end{proof}

\subsection{Viehweg--Zuo criterion}
In this section we will prove the main result of the paper, namely Theorem \ref{thm:main_intro}. The main tool is a criterion of Viehweg--Zuo \cite[Proposition 0.3]{VZ}, slightly generalized to treat the case of DM-stacks.
\begin{thm}\label{thm:VZ-criterion} Let $X/\mathbb{C}$ be a proper smooth Deligne--Mumford curve and\\
${f:A \rightarrow X}$ be a non-isotrivial family of $g$-dimensional principally polarized semi-abelian varieties. Denote by $S\subset X$ the subset of bad places for $f$, i.e., $V \colonequals X\setminus S$ is the largest open subset where the fibers are abelian. Assume that $S$ is finite.
\par Assume that the Arakelov inequality
\begin{equation} \label{eq:Arakelov}
    2 \deg(e^*\omega_{A/X}) \leq g\, \deg(\Omega_X^1(\log S))
\end{equation}
holds with equality, where $e: X\rightarrow A$ denotes the zero section.
Then the image of the induced map $U\rightarrow \mathcal{A}_g$ is a Shimura curve.
\end{thm}

\begin{proof}
Let $\phi: X' \rightarrow X$ be a finite map of smooth Deligne--Mumford curves \'etale outside $S$. Let $d\in \mathbb{Q}$ be the degree of $\phi$ as defined in \cite[Definition 1.15]{Vistoli}.
We claim that
\begin{equation}\label{eq:proofthmVZ}
\begin{aligned}
    \deg(e'^*\omega_{A'/X'})&= d \cdot \deg(f_*\omega_{A/X})\\
    \deg(\Omega_{X'}^1(\log S'))&= d \cdot\deg(\Omega_X^1(\log S))
\end{aligned}
\end{equation}
where $f': A' \rightarrow X'$ is the base-change of $f$, the map $e': X'\rightarrow A'$ is the zero-section of $f'$, and $S'=\phi^{-1}(S)$. In fact, the first equation in (\ref{eq:proofthmVZ}) is even true without assumptions on the ramification behavior of $\phi$ and follows from the isomorphism $e'^*\omega_{A'/X'} \cong \phi^* e^*\omega_{A/X}$. The second equation in \ref{eq:proofthmVZ} is well known for schemes (Riemann--Hurwitz). We are going to provide an argument that works for all proper smooth DM-curves. Indeed, we have the short exact sequence
$$0\rightarrow \phi^*\Omega_X^1 \rightarrow \Omega^1_{X'} \rightarrow \Omega_{X'/X} \rightarrow 0\,.$$
Because the ground field is algebraically closed of characteristic $0$, we know that locally around a ramification point $x\in X$ the map $\phi$ is of the form
$$[\mathbb{A}^1/G']\rightarrow [\mathbb{A}^1/G],\quad t\mapsto t^e$$
where $e=e_x$ is the ramification index at $x$.
Therefore, the stalk $\Omega_{X'/X,x}$ is isomorphic to $\frac{k[t]}{t^{e-1} k[t]}$. To see that
$$\deg\hspace{-0.1em}\left(\phi^*\hspace{-0.1em}\left(\Omega_{X}^1\hspace{-0.1em}\left(\log S\right)\right)\right)=\deg\hspace{-0.1em}\left(\Omega^1_{X'}\hspace{-0.1em}\left(\log S'\right)\right)$$
one can now argue similarly to the case of schemes: on the lefthand side, the point $x$ contributes $\frac{e}{|G'|}$. On the righthand side $x$ contributes $\frac{(e-1)+1}{|G'|}$, where the $(e-1)$ comes from comparing $\Omega^1_X$
 with $\phi^*\Omega^1_{X'}$ and the $1$ comes from $\log S'$. As a consequence we obtain 
 $$
 \deg(\Omega_{X'}^1(\log S'))=\deg\hspace{-0.1em}\left(\phi^*\left(\Omega_{X}^1\left(\log S\right)\right)\right)  = d \deg(\Omega_X^1(\log S)) \, ,
 $$
 which proves the claim.
\par Next, we claim that there exists a finite map $\phi: X' \rightarrow X$, \'etale outside $S$ such that $X'$ is a smooth scheme. First, we observe that $V=X\setminus S$ is hyperbolic, that is, the universal cover of $V^\text{an}$ is the upper half-plane (one can argue as in \cite[Proposition 1]{Beauville}). (For more discussion on hyperbolic DM-stacks, see \cite{BehrendNoohi}.) In particular, \cite[Corollary 7.7]{BehrendNoohi} implies that there is a smooth irreducible $\mathbb{C}$-scheme $V'$ and a finite \'etale map $V'\rightarrow V$. Let $X'$ be the smooth projective irreducible curve over $\mathbb{C}$ having $V'$ as an open dense subset. Now the map $V' \rightarrow V$ does not always extend to a map $X' \rightarrow X$. However, because of the properness of $X$ \cite[Th\'eor\`eme 7.3]{LaumonMB}, the composition $X'' \rightarrow X' \dashrightarrow X $ will extend to a regular map provided that $X''\rightarrow X'$ is a map sufficiently ramified at the points $S'=X'\setminus U'$. (This argument is analogous to \cite[Lemma V.5.1]{Hartshorne}.) Now the hyperbolicity of $V'$ implies the existence of a finite Galois map $X''\rightarrow X'$ \'etale outside $S'$ such that the ramification indices are sufficiently divisible to ensure that the composition $X'' \rightarrow X' \dashrightarrow X $ is a regular map. Then the map $\phi:X'' \rightarrow X$ is \'etale outside $S$ and $X''$ is a scheme. This proves the claim.
\par Using Equations \ref{eq:proofthmVZ} and the map $\phi$ constructed above we can now reduce to the case where $X$ is a scheme. In this case the theorem is known and the references to the proofs are as follows. The inequality \ref{eq:Arakelov} is proven in \cite[Lemme 3.2]{DeligneAra}. If the inequality \ref{eq:Arakelov} holds with equality, then we can apply \cite[Proposition 0.3]{VZ}, in the case where, in their notation, $g_0=g$. The theorem follows.
\end{proof}
\begin{rem}
     \par As pointed out in \cite[Remark 5.9]{MoonenCrvs} there is a mistake in the article of Viehweg--Zuo. In \cite[Theorem 0.7]{VZ} they attempt to classify Shimura curves where the Arakelov inequality \ref{eq:Arakelov} does not hold with equality, for example PEL-Shimura curves in $\mathcal{A}_{2g},\, g\geqslant 2$ with endomorphism algebra a degree $2g$ CM field. Moonen provides an explicit example not covered by \cite[Theorem 0.7]{VZ}. 
    The proof of \cite[Proposition 0.3]{VZ}, which is the only result of Viehweg--Zuo used in the present paper, is correct as written.
\end{rem}

Our next goal is to construct families of principally polarized abelian varieties over $\mathcal{X}(2,3,7)$ (resp., $\mathcal{X}(2,3,9)$) that are (up to descent and twist) the relative Jacobians of the families \ref{eq:237} (resp., \ref{eq:239}). For this purpose we have the following lemma.
\begin{lem}\label{lem:stackextend}
Let $p,q,r$ and $U \colonequals \PP^1 \setminus \{0, 1, \infty\}$ be as before. Assume in addition that $p$ is even. Let $A\rightarrow U$ be a family of abelian varieties. Assume that:
\begin{enumerate}
    \item[i)]  $A$ has potentially good reduction at $0$ and good reduction after a totally ramified extension of degree $2p$, say $\QQ\llbracket t^{1/2p} \rrbracket / \QQ\llbracket t \rrbracket$. Moreover, we ask that the inertia group of the subextension $\QQ\llbracket t^{1/2p} \rrbracket / \QQ\llbracket t^{1/p} \rrbracket$ acts by $[-1]$;
    \item[ii)] $A$ has potentially good reduction at $1$ and good reduction after a totally ramified extension of degree $r$; and
    \item[iii)] $A$ has potentially good reduction at $\infty$ and good reduction after a totally ramified extension of degree $q$.
    
\end{enumerate}
Then there is a family of abelian varieties $\mathcal{A} \rightarrow \mathcal{X}(p,q,r)$ extending $A$.

\end{lem}
\begin{proof}
By Lemma \ref{lem:UBZ2} we have $\mathfrak{U}\cong U\times B\mathbb{Z}/2$. Since every abelian variety has the automorphism $[-1]$, the family $A$ descends to $\mathfrak{U}$.
\par Next, we want to understand when the family can be extended from the open substack $\mathfrak{U}\subset \mathcal{X}(p,q,r)$ to the three points $t^{-1}(0,1, \infty)$. Let us first look at $0$. Indeed, the question is local at $0$ and therefore we can work with the stack $\mathbb{P}(2q, 2p)$. Consider the following commutative diagram, where $a,b \in \mathbb{Z}$ are chosen to satisfy $ap-bq=1$. For reasons that will become clear later, we will choose $a$ to be even. (This is possible because $q$ is odd and $a, b$ can be replaced by $a+q, b+p$.)
\begin{equation}\label{eq:cartdiagstackextend1}
\begin{tikzcd}
\mathbb{G}_m \arrow{r}{s\mapsto s^{2p}}\arrow{d}[swap]{s\mapsto s^{2}}& \mathbb{G}_m \arrow{d}{s\mapsto (s^a: s^b)}\\
\mathbb{A}^1 \arrow{r}[swap]{x\mapsto (x : 1)}& \mathbb{P}(2q, 2p)
\end{tikzcd}
\end{equation}
Notice that the right vertical map is the same as the map in the proof of Lemma \ref{lem:UBZ2}. We claim that diagram (\ref{eq:cartdiagstackextend1}) is cartesian. Indeed, the fiber product $
\mathbb{A}^1 \times_{\mathbb{P}(2q, 2p)} \mathbb{G}_m$ is a closed subscheme of $\mathbb{A}^1 \times \mathbb{G}_m \times \mathbb{G}_m$ with coordinates $x,s,\lambda$:
$$ \mathbb{A}^1 \times \mathbb{G}_m \times \mathbb{G}_m=\left\lbrace x \lambda^{2q}=s^{a},\, \lambda^{2p}=s^{b}\right\rbrace\,.
$$
Here the projection maps are $(x,s, \lambda) \mapsto x$ and $(x,s, \lambda) \mapsto s$, respectively. Now, since $p$ is even, the equation $ap-bq=1$ implies that $b$ and $2p$ are coprime. Consequently, the fiber product $\mathbb{A}^1 \times_{\mathbb{P}(2q, 2p)} \mathbb{G}_m$ identifies with $\mathbb{G}_m$ and the maps onto the two factors are the ones in diagram \ref{eq:cartdiagstackextend1}. 
\par Next, we point out that the map $\mathbb{A}^1 \rightarrow \mathbb{P}(2q, 2p)$ is \'etale. Thus, to check whether the family $A$ extends to $0$, it suffices to check on $\mathbb{A}^1$. For this purpose, we compute the descent datum along the left vertical map in the diagram (\ref{eq:cartdiagstackextend1}). Indeed, $\mathbb{G}_m\times_{ \mathbb{A}^1, \, s\mapsto s^2} \mathbb{G}_m \cong \mathbb{G}_m \times \mathbb{Z}/2$ and the natural map
$$\mathbb{G}_m\times_{ \mathbb{A}^1, \, s\mapsto s^2} \mathbb{G}_m \rightarrow \mathbb{G}_m\times_{ \mathbb{P}(2q, 2p)} \mathbb{G}_m\cong \mathbb{G}_m \times \mathbb{Z}/2$$
identifies with the map $(s\mapsto s^{2p})\times \id_{\mathbb{Z}/2}$. Therefore, since descent data are compatible under pull-backs, we see that the descent datum for the family $A$ along $\mathbb{G}_m \rightarrow \mathbb{A}^1, s\mapsto s^2$ is a quadratic twist. This shows i).
\par To show ii) we need to understand when the family $A$ extends to the point $t^{-1}(1)$. Now, Zariski locally around $t^{-1}(1)$ the map $\mathcal{X}(p,q,r) \rightarrow \mathbb{P}(2q, 2p)$ can be identified with
$$[\mathbb{A}^1/(\mu_{r}\times \mathbb{Z}/2)] \rightarrow [\mathbb{A}^1/(\mathbb{Z}/2)],\, x\mapsto x^r$$
where $\mu_r$ acts on $\mathbb{A}^1$ by multiplication and $\mathbb{Z}/2$ acts trivially. This shows that $A$ extends to the point $t^{-1}(1)$ provided that $A$ has potentially good reduction after a totally ramified extension of degree $r$.
\par The proof of iii) is similar to i). This time we look at the cartesian diagram
\begin{equation}\label{eq:cartdiagstackextend2}
\begin{tikzcd}
\mathbb{G}_m \arrow{r}{s\mapsto s^{q}}\arrow{d}[swap]{s\mapsto s}& \mathbb{G}_m \arrow{d}{s\mapsto (s^a: s^b)}\\
\mathbb{A}^1 \arrow{r}[swap]{x\mapsto (1 : x)}& \mathbb{P}(2q, 2p).
\end{tikzcd}
\end{equation}
Here we used that $a$ is even. Now it follows that $A$ extends to the point $t^{-1}(\infty)$ on $\mathcal{X}(p,q,r)$ if and only if $A$ has potentially good reduction at $\infty$ after a totally ramified extension of degree $q$.
\end{proof}


We now prove that the families \ref{eq:237}, \ref{eq:239} satisfy~\Cref{lem:stackextend}, and are Shimura families. In order to do so, we compute the semistable reduction at the three singular fibers $\{0,1,\infty\}$ for each family. We chose to include these calculations, as these semistable fibers are interesting in their own right. For example, after \Cref{thm:main} is shown, one knows that all these curves have complex multiplication: Indeed, since they have extra automorphisms but the generic endomorphism algebra of a Mumford type Shimura curve is $\Z$, their moduli points must lie in a zero dimensional Shimura subvariety; hence they have CM. More precisely,  we have the following corollary to \Cref{prop:Shimura} and \Cref{thm:main}, which are proved later.

\begin{cor}\label{cor:endos}
    The geometric endomorphism algebras $\End^0_{\overline{\Q}}$ of the semistable fibers of $C_{n,t}, \, n=7,9$ at $t_0 = 0, 1, \infty$ are given by the following table:
    \begin{table}[h]
        \centering
        \begin{tabular}{c|c|c}
             $t_0$ & $n=7$ & $n=9$ \\
             \hline
             $0$ &  $\Q(i)\oplus \Q(\zeta_7)^+(i)$&  $\Q(i)\oplus \Q(\zeta_9)^+(i)$\\
             $1$ &  $\Q(\sqrt{-7})\oplus \Q(\zeta_7)$ &  $\Q(\zeta_3)\oplus \Q(\zeta_9)$\\
             $\infty$ & $\Q(\zeta_3)\oplus \Q(\zeta_7)^+(\zeta_3)$& $\Q(\zeta_3)\oplus \Q(\zeta_9)$ 
        \end{tabular}
        \label{tab:endos}
    \end{table}
\end{cor}
\begin{proof}
     Let $A_0$ be the principally polarized abelian variety corresponding to the point $0\in \calX(2,3,7)$, i.e., the Jacobian of the semistable fiber of $C_{7,t}$ at $t_0=0$. Denote by $h: \mathbb{C}^* \rightarrow G(\R)$ the group homomorphism corresponding to $A_0$ as in section \ref{subsec:mumford_defn}. Notice that $h$ is not unique because of the action of the arithmetic group $\Gamma$. However, since $0\in \calX(2,3,7)$ is the elliptic point with an inertia jump of order $2$, we can choose $h$ so that it is centralized by the element $\delta_a = i  \in B_7$ (where $B_7 $ is as in \Cref{subsec:triangle}). Denoting by $L$ the subfield of $B_7$ generated by $i$, this implies that the image of $h$ is contained in the torus $T_L:=\Res_{L/\Q} \mathbb{G}_m \cap G$. Therefore, $T_L$ equals the Mumford-Tate group of $A_0$. Since $\calX(2,3,7)$ is a Shimura curve of Mumford type, we may use the results of Noot \cite[Section 3.2]{NootMumford} (see also \Cref{subsec:CM_Mumford}) to determine the endomorphism algebra of $A_0$. The other cases are analogous.
\end{proof}

\begin{rem}
    The corollary implies that all these abelian varieties are isogenous to a product of an elliptic curve and a threefold. As we will see below, for the fibers at $1$ this is even a decomposition as a product of principally polarized abelian varieties because the semistable fiber is reducible.

    The two entries in the table for $n = 9$ and $t_0=1,\infty$ are the same because $\Q(\zeta_9)=\Q(\zeta_9)^+(\zeta_3)$. However, the endomorphism algebras are different orders in $\End^0$.
\end{rem}

\medskip

Let us now compute the semistable reduction of the fibers at $0,1,\infty$ of the family $C_{7,t}$.
For completeness we present the changes of variables leading to the semistable reduction, but they are also available in the \Magma{} files \verb|237.m| and \verb|239.m|, in the GitHub repository \cite{Mumford-Github}.

Recall that the family $C_{7,t}$ is given by \Cref{eq:237}:
\begin{equation}
       \begin{aligned}
          C_{7,t}: y^2 &= t\left(\left(t - \frac{27}{16}\right) x^{10} - \frac{567}{64} x^{9} - \frac{189}{4} t x^{8} + \left(-84 t^{2} - \frac{189}{4} t\right)\right. x^{7}\\ 
          &\left.\quad - 189 t^{2} x^{6} - \frac{189}{2} t^{2} x^{5} + 84 t^{3} x^{4} + 108 t^{3} x^{3} - 28 t^{4} x\right) \, .
       \end{aligned}
   \end{equation}
   One can check using the file \verb|discriminant.m| from \cite{Github-Mumford} that the discriminant of the family $C_{7,t}$ is equal to $2^{36}\cdot 3^{36} \cdot 7^{10} \cdot t^{54}(t-1)^{12}$. We will now compute the semistable reduction~\cite{Liu, Kunzweiler} of the family at the singular fibers $\{0, 1, \infty\}$.

At the same time, we also want to keep track of the contributions to the lefthand side of the Arakelov inequality. For the latter purpose, we consider the rational section of $e^*\omega_{C_{7,t}/\mathbb{P}^1}$ \[\omega_0 := \frac{dx}{2y} \wedge x \frac{dx}{2y} \wedge \cdots \wedge  x^{g-1} \frac{dx}{2y},\] and note that for any integers $i$ and $j$, the transformation $(x, y) = (t^ix', t^jy')$ acts on $\omega_0$ by multiplication by $t^{\frac{g(g+1)}{2}i-gj}$. With that in mind, we will later be able to compute the degree of $e^*\omega_{A/\mathcal{X}(2,3,7)}$ from the order of zeros and poles of the rational section $\omega_0$ after showing that the family $\Jac(C_{7,t})\rightarrow \mathbb{P}^1$ descends to a family $A\rightarrow \mathcal{X}(2,3,7)$.

\begin{lem}\label{lem:red7_0}
    The semistable reduction of $C_{7,t}$ at $0$, obtained after a totally ramified extension of degree $4$, is \[y^2 = x(x^8 + 16/3x^6 + 32/3x^4 - 256/21x^2 + 256/81).\] It can also be realized as the reduction of a quadratic twist of the family obtained after a totally ramified extension of degree $2$.

    Moreover, the local contribution to the lefthand side of the Arakelov inequality is $-6$.
\end{lem}

\begin{proof}
    To compute the reduction at $0$ we perform the substitution 
            $t=s^2,\, x=s x',\, y=s^5 y'$. This transforms the equation into
            \begin{equation}
                \begin{aligned}
                    y'^2 &= s \Bigg( \left(s^3 - \frac{27}{16} s\right) x'^{10} - \frac{567}{64} x'^{9} - \frac{189}{4} s x'^{8} + \left(-84 s^2 - \frac{189}{4}\right) x'^{7}\\ 
                    &\quad - 189 s x'^{6} - \frac{189}{2} x'^{5} + 84 s x'^{4} + 108 x'^{3} - 28 x' \Bigg)\, ,
                \end{aligned}
            \end{equation}
            which has good reduction after a quadratic twist verifying condition i) of Lemma~\ref{lem:stackextend}.
            The contribution to the lefthand side of the Arakelov inequality for that fiber is thus $\frac{10\cdot 2 - 4\cdot 11}{4} = -6$.

\end{proof}

\begin{lem}\label{lem:red7_1}
    The semistable reduction of $C_{7,t}$ at $1$, obtained after a totally ramified extension of degree $7$, is the union of
    \begin{enumerate}
        \item the elliptic curve \[y^2 = x^3 - \frac{45}{28}x + \frac{27}{28},\text{ and }\]
        \item the genus $3$ hyperelliptic curve \[y^2 = x^7-3.\]
    \end{enumerate}
    Moreover, the local contribution to the lefthand side of the Arakelov inequality is $-\frac{9}{7}$.
\end{lem}

\begin{proof}
    By setting $t=1$, we see that the elliptic curve $y^2 = x^3 - \frac{45}{28}x + \frac{27}{28}$ with $j$-invariant $-3375$ is a component of the semistable reduction. 
    To find the missing component, we first substitute $t = t'+1$ to reduce to the case where $t' = 0$. After the transformation $(x,y)=\left(\frac{2}{x'-1}, \frac{y'}{(x'-1)^5}\right)$, we substitute $t' = s^7$, $x'=s^2x''$, and $y' = s^7y''$, and after specializing to $s = 0$, the genus $3$ curve $y^2 = x^7 -3$ appears as the other component. Thus, condition ii) of Lemma~\ref{lem:stackextend} is satisfied.

    This time, the contribution to the Arakelov inequality is \[\frac{6\cdot 2 - 3\cdot 7}{7} = -\frac{9}{7}.\]
\end{proof}

\begin{lem}\label{lem:red7_inf}
     The semistable reduction of $C_{7,t}$ at $\infty$, obtained after a totally ramified extension of degree $3$, is given by \[y^2 = x(x^{9} - 84x^6 + 84x^3 - 28).\] 
     Moreover, the local contribution to the lefthand side of the Arakelov inequality is $\frac{22}{3}$.
\end{lem}
\begin{proof}
    For the fiber at $\infty$, the substitution $t = s^3$, 
    $x = sx'$, $y = s^{8}y'$ yields after specialization to $s = \infty$ the smooth curve
    \[y^2 = x(x^{9} - 84x^6 + 84x^3 - 28),\]
    hence verifying condition iii) of Lemma~\ref{lem:stackextend}.
    The final contribution to the Arakelov inequality is \[-\frac{10\cdot 1 - 4\cdot 8}{3} = \frac{22}{3}.\] 
\end{proof}

We now consider the family $C_{9,t}$. Recall that this family is given by the system of equations \ref{eq:239}:
\begin{equation}
  \begin{aligned}
    C_{9,t}: 0 &= X Z-Y^{2}\\
    0 &= 3W^3 + t(t - 1)\Big( ( 5X^2 + 6XY + 2tYZ + 3tZ^2 )3W \\
    &+ (-2t + 9)X^3 + 22tX^2Y + 21tX^2Z + (-14t^2 + 18t)XYZ \\
    &+ t^2XZ^2 + 6t^2YZ^2 + (-3t^3 + 6t^2)Z^3 \Big) \, .
  \end{aligned}
  \end{equation} 

It suffices to restrict to the affine open given by $Z = 1$. This allows us to set $X = Y^2$ so we can reduce to looking at a single equation in $\mathbb{A}^2$:
\begin{equation}
  \begin{aligned}
    F_9: -3t^5 &- 14t^4Y^3 + t^4Y^2 + 6t^4Y + 9t^4 - 2t^3Y^6 + 22t^3Y^5 +
    21t^3Y^4 \\
    &+ 32t^3Y^3 - t^3Y^2 + 6t^3YW - 6t^3Y + 9t^3W - 6t^3 +
    11t^2Y^6 \\
    &-22t^2Y^5 + 15t^2Y^4W - 21t^2Y^4 + 18t^2Y^3W -
    18t^2Y^3 - 6t^2YW \\
    &- 9t^2W - 9tY^6 - 15tY^4W - 18tY^3W + 3W^3.
  \end{aligned}
  \end{equation}

    One can check using the file \verb|discriminant.m| from \cite{Github-Mumford} that the discriminant of the family of genus $4$ curves $C_{9,t}$ is equal to $2^{72}\cdot 3^{34}\cdot  t^{52} (t-1)^{28}$.
    As with the previous family, we compute the semistable reduction~\cite{Liu, LLLGR} of $C_{9,t}$ at the singular points $\{0,1,\infty\}$, while also calculating the contributions to the lefthand side of the Arakelov inequality. This time, we choose the rational section \[\omega_0 := \frac{dY}{dF_{9, W}}\wedge\frac{WdY}{dF_{9, W}}\wedge\frac{Y dY}{dF_{9,W}}\wedge\frac{Y^2dY}{dF_{9, W}}.\]

The transformation $(Y, W) = (t^i\,Y', t^j\,W')$ acts on $\omega_0$ by multiplication by $t^{7i+5j}$. Next, dividing $F_9$ by $t^k$ replaces $\omega_0$ by $t^{4k}\, \omega_0$. 
\begin{samepage}
\begin{lem}\label{lem:red9_0}
    The semistable reduction of $C_{9,t}$ at $0$, obtained after a totally ramified extension of degree $4$, is the hyperelliptic curve \[y^2 = x(x^8 - 4x^6 + 6x^4 - 44/27x^2 + 1).\]Moreover, after a totally ramified extension of degree $2$ a quadratic twist of the Jacobian of $C_{9,t}$ has good reduction.
    The local contribution to the Arakelov inequality is $-6$.
\end{lem}
\end{samepage}
\begin{proof}
    Once we substitute 
    \[t = s^2, Y = s Y', W = s^3(-W'-1/3Y')-s^2,\] we can divide $F_9(s,Y',W')$ by $s^8$. We find
    {\footnotesize
    \begin{align*}
        (Y^3-W)^2 =\  &s\left(\frac{17}{9}Y^5 - \frac{5}{3} Y^4W - \frac{125}{81}Y^3 + \frac{1}{9}Y^2W + \frac{1}{3}YW^2 + \frac{1}{3}Y + \frac{1}{3}W^3 + W\right)\\
            &+s^2g(Y,W).
    \end{align*}
    }
    Specializing to $s = 0$ gives the equation $(Y^3-W)^2 = 0$, so this type of equation can be seen as an analogue of \cite[Section 2]{LLLGR} in genus $4$. By setting $s = s'^2, x = Y, y = \frac{Y^3-W}{s'}$, we resolve the singularities, and find that the hyperelliptic curve of genus $4$
    \[y^2 = x(x^8 - 4x^6 + 6x^4 - 44/27x^2 + 1)\]
    is the semistable reduction at $0$.
    Thus condition i) of Lemma~\ref{lem:stackextend} is satisfied.
        
    The contribution to the lefthand side of the Arakelov inequality is \[\frac{(7\cdot 2+5\cdot  6-4\cdot  16)-4}{4} = -6,\]
    where the term $-4$ can be understood as follows. We have $dx = dY$, and $y = \frac{Y^3-W}{s'}$, therefore comparing $\omega_0$ with $\frac{dx}{y}\wedge\frac{xdx}{y}\wedge \frac{x^2dx}{y}\wedge \frac{x^3dx}{y}$ then introduces an additional factor of $s'^{-4}$, arising from the division by $s'^{-1}$ in $y$.
\end{proof}

\begin{lem}\label{lem:red9_1}
    The semistable reduction of $C_{9,t}$ at $1$, obtained after a totally ramified extension of degree $7$, is the union of an elliptic curve isomorphic to $y^2 = x^3 - 1$ and the Picard curve $y^3=x(x^3+1)$.
    
    Moreover, the local contribution to the lefthand side of the Arakelov inequality is $-\frac{29}{9}$.
\end{lem}

\begin{proof}
    To compute the semistable reduction at 1 we first move $1$ to 0 by 
    replacing $F_9$ by the equation $F_{9, 1} = F_9(t+1).$ Considering the equation $F'_{9,1} = s^{-3}F_{9,1}(s^3, Y' - 1, sW')$ and reducing modulo $s$ gives us the equation $Y'^4\cdot(Y'^2 - 20/7Y' + 16/7) + W'^3$,
    which is an elliptic curve isomorphic to $y^2 = x^3 - 1$. Now looking at $F''_{9,1} = v^{-12}F'_{9,1} (v^3,v^3Y'',v^4W'')$  and reducing modulo $s$, we obtain the equation
    $16Y'^4 + 16Y' + 3W'^3,$
    showing that the Picard curve given by  $y^3=x(x^3+1)$ is another component. 
    This computation shows that condition ii) of Lemma~\ref{lem:stackextend} is verified.
            
    The contribution to the lefthand side of the Arakelov inequality is
    \begin{align*}
       &\frac{7\cdot 0+5\cdot 1-4\cdot3}{3} + \min \left(0, \frac{3+4-12}{9}\right)+\min \left(0, \frac{3+2\cdot 4-12}{9}\right)\\ 
       &\qquad +\min \left(0, \frac{ 2\cdot 3+4-12}{9}\right)+\min \left(0, \frac{3\cdot 3+4-12}{9}\right)=-\frac{29}{9}
    \end{align*}
    
\end{proof}

\begin{lem}\label{lem:red9_inf}
    The semistable reduction of $C_{9,t}$ at $\infty$, obtained after a totally ramified extension of degree $3$, is the genus $4$ curve with plane model given by \[-2Y^6 + 15Y^4W - 14Y^3 + 6YW + 3W^3 - 3=0.\]
    
    Moreover, the local contribution to the lefthand side of the Arakelov inequality is $\frac{28}{3}$.
\end{lem}

\begin{proof}
    To compute the semistable reduction at infinity, we make the substitution $(t, Y, W) \mapsto (s^3, sY', s^5W')$, and after dividing by $s^{15}$, specializing $s$ to $\infty$ gives us the equation
    $$-2Y'^6 + 15Y'^4W' - 14Y'^3 + 6Y'W' + 3W'^3 - 3.$$
    Since this equation is the affine equation of a smooth genus $4$ curve, we have shown that $C_{9,t}$ has potentially good reduction after an extension of ramification degree $3$ verifying criterion iii) of Lemma \ref{lem:stackextend}.
    The contribution of this fiber to the Arakelov inequality is $-\frac{7\cdot 1+5\cdot  5-4\cdot 15}{3} = \frac{28}{3}$.    
\end{proof}

Having completed all computations we are now ready to show that our families can be extended to families of abelian varieties over $\calX(2,3,7)$ and $\calX(2,3,9)$ respectively, and that for the extended families the Arakelov inequality~\eqref{eq:Arakelov} holds with equality.
\begin{prop} \label{prop:Shimura}
    The Jacobians of the families $C_{7,t}, C_{9,t}$ (Equations \ref{eq:237}, \ref{eq:239}) over $\mathbb{P}^1\setminus \{0,1, \infty \}$ satisfy the criterion of Lemma \ref{lem:stackextend} with $(p,q,r)=(2,3,7)$ (resp. $(2,3,9)$).
    \par Furthermore, for the extended families over $\calX(2,3,7)$ and $\calX(2,3,9)$, we have equality in the Arakelov inequality (\ref{eq:Arakelov}). In particular, the families are Shimura families. 
\end{prop}

\begin{proof}
    We proved that Lemma \ref{lem:stackextend} is verified for both our families (see Lemmata \ref{lem:red7_0}, \ref{lem:red7_1}, and \ref{lem:red7_inf}, resp. \ref{lem:red9_0}, \ref{lem:red9_1}, and \ref{lem:red9_inf})).
    Therefore there exist families $A_n\rightarrow \calX(2,3,n), \, n=7,9$ of principally polarized abelian varieties obtained from descending $\Jac(C_{n,t})$. We will show that these abelian schemes satisfy the Arakelov inequality with equality. 
    Let us consider the family $C_{7,t}$. From the contributions of the different fibers, we compute both sides of the inequality and find
    $$\deg(e^*\omega_{{A_7}/\calX(2,3,7)})= \frac{1}{2}\left(-6-\frac{9}{7}+\frac{22}{3}\right)=\frac{1}{42}$$
    where the $\frac{1}{2}$ accounts for generic inertia.
    On the other hand $\deg(\Omega^1_{\calX(2,3,7)})=\frac{1}{84}$ by  \cite[Equation 5.5.10]{VZB}, showing that we have equality in Arakelov's inequality for $\calX(2,3,7)$ and that \Cref{eq:237} is a Shimura family.
    \bigskip
        
    Similarly, for the family $C_{9,t}$, we find \[\deg(e^*\omega_{{A_9}/\calX(2,3,9)})=\frac{1}{2}\left(-6-\frac{29}{9}+\frac{28}{3}\right) =\frac{1}{18},\]
    showing once again that equality holds in Arakelov's inequality for $\calX(2,3,9)$, and that \Cref{eq:239} is a Shimura family.
\end{proof}
Henceforth we use the notation $\mathcal{X}_7 := \mathcal{X}(2,3,7),\mathcal{X}_9 := \mathcal{X}(2,3,9)$. This differs slightly from the terminology in \Cref{thm:main_intro} as we now include generic inertia.
\subsection{The Shimura curves are of Mumford type}
Now that we know that $\calX_7$ and $\calX_9$ are Shimura curves, we want to show that they are of Mumford type. The strategy will be to use the fundamental group to determine the quaternion algebra defining the Shimura curve. The fundamental group is determined by a result of Behrend--Noohi \cite[Proposition 5.6]{BehrendNoohi}.  There are two surjections.
\begin{align*}
    \pi_1(\calX_7) &\twoheadrightarrow \langle \sigma_2, \sigma_3, \sigma_7 \mid \sigma_2^2 = \sigma_3^3 = \sigma_7^7 = 1 \rangle = \Delta(2,3,7)\\
    \pi_1(\calX_9) &\twoheadrightarrow \langle \sigma_2, \sigma_3, \sigma_9 \mid \sigma_2^2 = \sigma_3^3 = \sigma_9^9 = 1 \rangle = \Delta(2,3,9).
\end{align*}
with kernel $\Z/2$.
We are now ready to prove the main result, Theorem \ref{thm:main_intro}.
\begin{thm}\label{thm:main}
The Shimura families constructed in Proposition \ref{prop:Shimura} are of Mumford type.
\end{thm}
\begin{proof}
Let $\mathcal{X}_n, \, n\in \{7,9\}$ be either $\mathcal{X}_7$ or $\mathcal{X}_9$. Then, as was already noted in the proof of Theorem \ref{thm:VZ-criterion}, the existence of a non-isotrivial family of principally polarized abelian varieties over $\calX_n$ implies that it is hyperbolic. Furthermore, by \cite[Proposition 7.5]{BehrendNoohi}, $\mathcal{X}_n^\text{an} \cong [\Gamma \backslash \H]$ where $\Gamma=\pi_1(\mathcal{X}_n)$ acts on $\H$ through a group homomorphism $\rho:\Gamma \rightarrow \PSL_2(\mathbb{R})$ with kernel $\Z/2$ and image $\Delta(2,3,n)$. The map $\rho$ is unique up to conjugation. Since $\mathcal{X}_n$ is a Shimura curve, there exists a quaternion algebra $B$ over a totally real field $K$ with the following properties: \begin{itemize}
    \item $B$ splits at exactly one real place of $K$, say $v$; and 
    \item for some (equivalently, any) maximal order $\mathcal{O}\subset B$, $\rho(\Gamma)$ is commensurable with $\Gamma(B, \calO)$ (see \Cref{subsec:triangle}).
\end{itemize}
Furthermore, \cite[Proposition 5]{TakeuchiFuchs} shows that $B$ is uniquely determined from $\rho(\Gamma)$. Therefore, the quaternion algebra $B$ is the one described in \Cref{subsec:triangle}; recall that its center $K=\Q(\zeta_n)^+$ has degree $3$.
\par We now show that $\calX_n$ is of Mumford type using the classification of Shimura curves in $\mathcal{A}_g$ due to Moonen \cite{MoonenCrvs}. Let $\Sh \subseteq \mathcal{A}_g$ be a Shimura curve and $(A, \lambda)$ be the principally polarized abelian variety over $\mathbb{C}$ corresponding to a non-CM point of $\Sh$. Assume first that $A$ is simple (for the general case, see \cite[Section 6.1]{MoonenCrvs}). Then, Moonen describes a dichotomy between Shimura curves for which $\End^0(A)=\End(A)\otimes_\mathbb{Z} \mathbb{Q}$ is of Albert type IV (i.e., the center is a CM field), and those for which it is not. The latter are non-PEL when $g\geqslant 4$ and, similar to Mumford's construction, are all obtained via the corestriction. Albert type IV Shimura curves can be PEL or non-PEL. Their construction is not based on the corestriction and thus is not a generalization of Mumford's construction.
\par Returning to the Shimura curve $\mathcal{X}_n$, let $(A, \lambda)$ be the principally polarized abelian variety corresponding to a non-CM point of $\mathcal{X}_n$. Suppose, for contradiction, that $A$ has a simple non-CM factor $A_0$ of Albert type IV. Let $F$ denote the center of $\End^0(A_0)$. By \cite[Proposition 4.12]{MoonenCrvs} there is a subfield $k_I\neq \mathbb{Q}$ of the Galois closure of $K$ such that $k_I$ embeds into $F$. In our case $K$ is cyclic of degree $3$  over $\mathbb{Q}$ and thus $k_I=K$. Since $K$ is totally real and $F$ is a CM-field, we have $[F:\mathbb{Q}]\geqslant 6$. However, since $A_0$ does not have CM, this implies $\dim(A_0)\geqslant 6$, contradicting the fact that $\dim(A)=4$.
\par Therefore, $A$ must have at least one simple factor $A_0$ not of Albert type IV. Then \cite[Proposition 4.8]{MoonenCrvs} applies and we are in case (I). Consequently ${\dim(A_0)= 2^{[K:\mathbb{Q}]-1}=4}$ and therefore $A\sim A_0$. In particular, $A$ is simple and $\mathcal{X}_n$ is a Mumford-type Shimura curve because of the conclusion (2) of loc. cit. This proves the theorem.    
\end{proof} 
\begin{rem}
    Alternatively, one could use \cite[Proposition 0.5]{VZ} to prove the previous theorem.    
\end{rem}

\begin{rem}\label{rem:LuZuoError}
Theorem \ref{thm:main} contradicts the claim of \cite[Theorem 1.2]{LuTanZuo} asserting the non-existence of (generalized) Mumford type Shimura curves in the non-hyperelliptic locus for $g\geqslant 4$. Their proof is based on slope inequalities for families of semistable curves. It appears that there are issues with some of the coefficients in these inequalities. For instance, in the earlier paper \cite[Theorem 5.2]{LuZuo} they consider a family $h: S\rightarrow B$ of semistable curves over a smooth curve $B$ which is obtained from a curve $C$ in the Torelli locus, so that $\pi: B\rightarrow C$ is $2:1$ ramified at some set of hyperelliptic points $\Lambda$. They claim that, under the assumption that $h_* \omega_{S/B}$ is slope semistable, one has the inequality
\begin{equation}\label{eq:LuZuo}
    \omega_{S/B}^2\geqslant \frac{5g-6}{g} \deg h_*\omega_{S/B}+2(g-2) |\Lambda| \, .
\end{equation}
(The full inequality in \cite[Theorem 5.2]{LuZuo} contains additional terms on the righthand side, but these are all positive.)
The $2$ in front of $(g-2)$ is based on the following fallacious argument: they consider the multiplication map $\rho: S^2 h_* \omega_{S/B}\rightarrow h_*\omega_{S/B}^{\otimes2}$ and claim that for points in $x \in \Lambda$, the length of the cokernel satisfies
$$
{\operatorname{length}_x \coker (\rho) \geqslant 2 \dim(\coker \rho|_{x})}
$$
because the vector bundle $h_* \omega_{S/B}$ is pulled back from $C$ (see \cite[Lemma 5.5]{LuZuo}). This is not correct, as one would have to assume that the map $\rho$ is pulled back from $C$, which is not always the case.
\par Indeed as we will show now, our family \Cref{eq:239} gives an explicit counterexample to the inequality \ref{eq:LuZuo}. Choose an \'etale map $\varphi:C\rightarrow \calX_9$ where $C$ is a curve. Denote $d=\deg(\varphi)$. Let $S\rightarrow B$ be the family of semistable curves obtained from \Cref{eq:239} where $B\rightarrow C$ is as above. Then by the proof of \Cref{prop:Shimura}
$$\deg h_*\omega_{S/B}=\frac{2d}{18},$$
Furthermore $|\Lambda|=\frac{d}{4}$ and $\omega_{S/B}^2$ must satisfy Noether's inequality\\
${\omega_{S/B}^2\leqslant 12 \deg(h_* \omega_{S/B})}$ \cite[Equation 2.3]{LuZuo}. Therefore inequality (\ref{eq:LuZuo}) would yield
$$
\frac{24d}{18}\geqslant \frac{14\cdot 2 d}{4 \cdot 18}+d = \frac{25d}{18}
$$
which is absurd. On the other hand the vector bundle $h_*\omega_{S/B}$ is slope semistable by \cite[p. 72]{LuZuo}.
\par Despite these issues the arguments of \cite{LuZuo} still show that there are no (generalized) Mumford type Shimura curves generically contained in the Torelli locus for $g \gg 0$. It would be interesting to find the optimal lower bound on $g$.
\end{rem}

\begin{rem}
    In \cite[Question 6.2]{MoonenOort}, the authors ask if one can construct a curve $C$ of genus at least $4$ such that $\Aut(C) = \{\id\}$ and such that $\Jac(C)$ is an abelian variety of CM type. Since $C_{9,t}$ is not hyperelliptic and Mumford-type abelian varieties have endomorphism algebra $\Z$, the generic automorphism group of a curve in the family $C_{9,t}$ is trivial. As there are infinitely many CM points on the Shimura curve $\calX_9$, this gives an affirmative, although non-explicit, answer to their question. 
    
    All the entries of \Cref{table:239} are heuristic examples of such curves. Since the equations were found numerically, it is possible---although unlikely, given the high precision of our computations---that the curves might in fact not have CM. An invariant theoretic argument confirmed that all these curves have trivial automorphism group.
\end{rem}

\section{Further questions} \label{sec:future}
In this section we present some related questions and potential directions for future research.

\subsection{Cycles}

As before, let $A$ be an abelian fourfold of Mumford type. Let $K$ be the associated totally real cubic field and $B$ be the $K$-quaternion algebra. A special feature of abelian varieties of Mumford type is the existence of exceptional Hodge classes not arising from endomorphisms of $A$. More precisely, there are two Hodge classes $\delta_1, \delta_2\in H^4(A^2, \Q)$ that cannot be written as linear combinations of products of divisor classes; see \cite[(5.9)]{Moonen_MumTate}. Furthermore, products of $\delta_1, \delta_2$, and divisor classes generate the Hodge classes in $H^*(A^2, \mathbb{Q})$.

As noted by Galluzzi \cite{Galluzzi-KugaSatake}, while Kuga \cite{Kuga} proved the Hodge conjecture for abelian varieties of Mumford-type, the Hodge conjecture for products of such varieties remains open. (See also \cite{MoonenZarhin}, where Moonen and Zarhin prove the Hodge conjecture for abelian varieties of dimension at most $5$.) Thus it is unknown whether the Hodge classes mentioned above are algebraic. The explicit equations for the families of Mumford type given in \Cref{thm:main_intro} invite the possibility of verifying the Hodge conjecture for these families by attempting to find explicit algebraic cycles realizing these Hodge classes.

When $A$ is the Jacobian of a curve $C$, then the exceptional Hodge classes $\delta_1, \delta_2$ can be realized in $H^1(C, \QQ)^{\otimes4}\subset H^4(C^4, \mathbb{Q})$. Therefore, one possibility for proving the algebraicity of $\delta_1, \delta_2$ for a fiber $C$ of one of the families from \Cref{eq:237,eq:239} is finding explicit surfaces in $C^4$ realizing these classes. One might also hope to find equations that work generically, realizing these classes for the entire family.



\subsection{K3 Surfaces}\label{subsec:K3}
Another possible strategy to establish the algebraicity of the Hodge classes $\delta_1, \delta_2$ uses K3 surfaces and the Kuga--Satake construction.
Central to the connection between K3 surfaces and abelian fourfolds of Mumford type are Hodge structures of K3 type. 
Recall that a Hodge structure of K3 type is a polarized rational Hodge structure $W$ of weight 2 such that $\dim_{\CC}V^{2,0} =1$. Naturally, the second cohomology group $H^2\hspace{-0.09em}(S, \ZZ)$ for any any K3 surface $S$ satisfies this property. 

Given a polarized Hodge structure of K3 type $W$, the Kuga--Satake construction produces a polarized Hodge structure of weight 1, i.e., a polarized abelian variety $A$ with the property that there is an inclusion of polarized Hodge structures
\begin{equation}\label{eq: Kuga--Satake}
W \hookrightarrow H^1\hspace{-0.1em}(A, \QQ) \otimes_{\Q}H^1\hspace{-0.1em}(A, \QQ)\,.
\end{equation}
We write $\KS(W)$ for this abelian variety $A$. The functor $\KS(\cdot)$ is essentially injective.

In \cite[Lemma 4.5]{Galluzzi-KugaSatake}, Galluzzi shows that with $\Lambda \colonequals H^1\hspace{-0.09em}(A,\Z)$ there is a sub-Hodge structure $W$ of $S^2 \Lambda$ of $K3$ type with Hodge numbers (1,7,1). The associated K3 surface of Picard rank $13$ has RM as $W$ has the following properties:
\begin{enumerate}
\item $K \subseteq \End_{\mathbb{Q}-\text{HS}}(W_{\Q})$; \cite[Proposition 4.7]{Galluzzi-KugaSatake}; and
\item there is a symmetric $K$-bilinear form $b$ on $W_{\Q}$ giving a polarization on $W$. The pair $(W_{\Q}, b)$ is isomorphic to the reduced trace pairing  of $B$, i.e.,
$(\{x\in B \mid \trd(x)=0\}, \quad (x, y) \mapsto \trd(x\cdot y))$
\end{enumerate}

Galluzzi further shows that $\KS(W)$ is isogenous to $X^{32}$. Applying this construction to our families $\calX_7$ and $\calX_9$ would give us two families of K3 surfaces with RM. It would be interesting to describe these families explicitly.

Suppose we knew the family of K3-surfaces and the algebraicity of the Hodge classes giving rise to the RM and the Kuga--Satake construction, i.e., the inclusion in \Cref{eq: Kuga--Satake}; then this would give another way of establishing the Hodge conjecture for $\Jac(C)^2$, where $C$ denotes any fiber of one of our families.

Examples of families of K3 surfaces with RM are described by van Geemen and Schütt in \cite{vanGeemenSchuett}. The families in \cite[5.3]{vanGeemenSchuett} $(n=7)$ (resp. 5.5, $n=9$) with RM by $\Q(\zeta_7)^+$ (resp. $\Q(\zeta_9)^+$) are particularly interesting candidates for Galluzzi's construction.  The Picard ranks $\rho$ of the van Geemen--Schütt families are $\rho = 4$ for $n=7$, and $\rho = 10$ for $n=9$. One might hope to find suitable subfamilies with Picard rank $13$ that realize Galuzzi's construction. 
 Also note that they describe a family in \cite[5.7]{vanGeemenSchuett} with $n=11$ that might be related to the $\mathcal{X}_{11}$ family of abelian varieties corresponding to the $(2,3,11)$ triangle group mentioned in Remark \ref{rem:2311}.

Another possible starting point for finding Galluzzi's family of K3 surfaces is a construction of Paranjape \cite{Paranjape}. He studies K3 surfaces $S$ which are the desingularization of a double cover of the plane branched along six lines. The Kuga--Satake variety associated to the second cohomology group of $S$ is shown to be isogenous to $(\Prym \pi)^4$ where $\pi:C\to E$ is a cyclic four-to-one cover between a genus 5 curve $C$ and an elliptic curve $E$. Paranjape gives a recipe to write down the K3 surface when such a cover is given. We can apply this construction to the fiber at $0$ of our families \ref{eq:237}, \ref{eq:239} which are curves with automorphism group isomorphic to $\mathbb{Z}/4$. This is the degenerate case of Paranjape's construction, where $E$ degenerates to a rational curve. It would be interesting to see if Paranjape's K3 surface can be deformed explicitly in order to realize Galluzzi's family.

\subsection{Hypergeometric functions}\label{subsec:hypgeom}

The connection between hypergeometric functions and triangle groups is classical. As the Shimura curves $\calX_7$ and $\calX_9$ parametrizing the families \ref{eq:237} and \ref{eq:239} are defined using triangle groups, we expect that the periods of the families can be related to hypergeometric functions. We begin by recalling the definition of hypergeometric functions.

Let $U$ be the unit ball $\{t\in\C : |t|< 1\}$. The hypergeometric function $F(a,b;c;t): U \to \C$ is given by the hypergeometric series
\[ \sum_{n=0}^\infty \frac{(a)_n(b)_n}{(c)_n}\frac{t^n}{n!}\]
where $(m)_n = \prod_{i=0}^n(m+i)$. In the case where $b$ and $c-b$ are both non-negative the hypergeometric function can also be given using the integral representation
\[
    F(a,b;c;t) = \frac{\Gamma(b)\Gamma(c-b)}{\Gamma(c)} \int_1^\infty x^{a-c} (x-1)^{c-b-1} (x-t)^{-a} \, dx \, .
\]
Assume $a, b,$ and $c$ are rational and write
\[\mu_1 := c-a = \frac{n_1}{d},\hspace{1em} \mu_2 := 1+ b -c = \frac{n_2}{d},\hspace{1em} \mu_3 := a = \frac{n_3}{d}\]
as reduced fractions. Since
\[
    x^{a-c} (x-1)^{c-b-1} (x-t)^{-a}=x^{-\mu_1}(x-1)^{-\mu_2}(x-t)^{-\mu_3},
\]
then the integral is a period integral of the superelliptic curve given by 
\[y^d = x^{n_1} (x-1)^{n_2} (x-t)^{n_3}.\]
 There is a classical result of Hermite, Pochhammer, Schl\"afli, and Schwarz which has the following modern interpretation \cite{DeligneMostow}: for certain values of $a,b,c$ the cohomology of this family of supererelliptic curves has a piece with monodromy $\Delta(p,q,r)$. The corresponding values of $\mu_1, \mu_2, \mu_3$ can be found in Deligne-Mostow \cite[14.3.1]{DeligneMostow}. For $(p,q,r) = (2,3,7)$  we get 
\begin{align*}
    \mu_1 &= \frac{1}{2}\left(1-\frac{1}{p}-\frac{1}{q}+\frac{1}{r}\right)= \frac{13}{84},\\ 
    \mu_2 &= \frac{1}{2}\left(1-\frac{1}{p}+\frac{1}{q}-\frac{1}{r}\right) =  \frac{29}{84},\\
    \mu_3 &= \frac{1}{2}\left(1+\frac{1}{p}-\frac{1}{q}-\frac{1}{r}\right) = \frac{43}{84},
\end{align*}
corresponding to the curve
$$
    y^{84}=x^{13} (x-1)^{29}(x-t)^{43} \, .
$$
For this equation the elliptic points of order $p,q,r$ are at $0, 1, \infty$, respectively. However, for our families \ref{eq:237}, \ref{eq:239} the elliptic points are ordered $0, \infty, 1$, and hence we consider the curve 
$$X_{7,t}: y^{84}=x^{13} (x-1)^{43}(x-t)^{29}\,.$$
instead.
Similarly, setting $(p,q,r) = (2,3,9)$ gives us
\[
\mu_1 = 5/36, \quad \mu_2 = 13/36, \quad \mu_3 = 19/36
\]
and the curve
\[X_{9,t}: y^{36}=x^{5} (x-1)^{19}(x-t)^{13}.\]

It is natural to expect a connection between the periods of $X_{n,t}$ and $C_{n,t}$ for $n=7,9$. Before we formulate such a conjectural relationship, we recall an important idea of Shimura \cite{ShimuraCanonical} (see also \cite[§6]{Deligne}). Let $K$ and $B$ be as in \Cref{subsec:mumford_defn}, and $F$ be a CM extension of $K$ splitting $B$. Shimura observed that two types of Shimura curves become isomorphic over $\C$: the Mumford-type Shimura curves associated to $B$ and the PEL Shimura curves parametrizing polarized $6$-dimensional abelian varieties with a certain action of $F$. In the language of Deligne \cite{Deligne} this is equivalent to saying that the two Shimura data have the same adjoint Shimura datum.
\par In Shimura's construction, there is no unique choice for $F$. However, in our setting only one $F$ will provide the link to hypergeometric functions. We explain how to find it for the $(2,3,7)$ family; the $(2,3,9)$ case is analogous. Consider the family
$$X_{7,t}: y^{84}=x^{13} (x-1)^{43}(x-t)^{29}\,.$$
Define the Prym variety $\Prym(X_{7,t})$ as the quotient of $\Jac(X_{7,t})$ by the Jacobians of the subcovers of the projection $X_{7,t} \to \PP^1,\, (x,y)\mapsto x$. Then $\Prym(X_{7,t})$ is a $24$-dimensional (non-principally) polarized abelian variety with an action of $\Q(\zeta_{84})$. Furthermore, Cohen and Wolfart \cite[p.~100]{CohenWolfart} proved that $\Prym(X_{7,t})$ is isogenous to $A_{7,t}^4$ where $A_{7,t}$ is a $6$-dimensional abelian variety having an action of a degree $6$ CM-field $F_7\subset \Q(\zeta_{84})$. To find $F_7$ we inspect the action of $\Q(\zeta_{84})$ on the differential forms of $\Prym(X_{7,t})$. The eigenspaces have dimensions
\begin{table}[h!]
\centering
    \begin{tabular}{c||c|c|c|c|c|c|c|c|c|c|c|c}
          $i$ & 1 & 5 & 11 & 13 & 17 & 19 & 23 & 25 & 29 & 31 & 37 & 41\\
         \hline
         $d_i$ & 1& 2& 2& 1& 2& 2& 2& 2& 1& 2& 2& 1 \\
         \multicolumn{11}{c}{\vspace{0.1cm}}\\
            $i$ & 43 & 47 & 53 & 55 & 59 & 61 & 65 & 67 & 71 & 73  & 79 & 83\\
         \hline
         $d_i$ & 1& 0& 0& 1& 0& 0& 0& 0& 1& 0& 0& 1    
    \end{tabular}
\end{table}
where the first row describes a character $\chi_i: \zeta_{84} \mapsto \zeta_{84}^i$ and the second row gives $d_i \colonequals \dim H^0\hspace{-0.09em}(\Prym(X_{7,t}), \Omega^1)^{\chi_i}$. Note that the dimensions satisfy $d_{84-i} = 2 - d_i$ due to Hodge duality. 

The numbers $d_i$ have the property $d_i=d_{41\, i}= d_{55\, i}$. This implies that $F_7$ is the fixed field of the field automorphisms $\sigma_1, \sigma_2: \Q(\zeta_{84}) \rightarrow \Q(\zeta_{84}),\, \sigma_1(\zeta_{84})=\zeta_{84}^{41},\, \sigma_2(\zeta_{84})=\zeta_{84}^{55}$, and thus ${F_7=\Q(\zeta_7)^+(\sqrt{-21})}$. Motivated by the above observation of Shimura, we propose the following conjecture:
\begin{conj} \label{conj:HGM237}
Let $E_7$ be an elliptic curve with CM by $\QQ(\sqrt{-21})$. Then for any fiber
$$X_{7,t}: y^{84}=x^{13} (x-1)^{43}(x-t)^{29}$$
of the hypergeometric family, the Hodge structure
\[\left(H^1\hspace{-0.1em}(X_{7,t}, \Q)^{\otimes 3} \otimes_{\mathbb{Q}} H^1\hspace{-0.1em}(E_7, \Q)^{\otimes 2}\right)\!(2)\]
has a summand isomorphic to $H^1\hspace{-0.1em}(C_{7,t}, \Q)$.
\end{conj}
Similar considerations for the $(2,3,9)$ family yield $F_9=\Q(\zeta_9)^+(i)$ and lead us to the following conjecture.
\begin{conj} \label{conj:HGM239}
Let $E_9$ be an elliptic curve with CM by $\QQ(i)$. Then for any fiber
$$X_{9,t}: y^{36}=x^{5} (x-1)^{19}(x-t)^{13}$$
of the hypergeometric family, the Hodge structure
\[\left(H^1\hspace{-0.1em}(X_{9,t}, \Q)^{\otimes 3} \otimes_{\mathbb{Q}} H^1\hspace{-0.1em}(E_9, \Q)^{\otimes 2}\right)\!(2)\] has a summand isomorphic to $H^1\hspace{-0.1em}(C_{9,t}, \Q)$.
\end{conj}
Assuming these conjectures hold, then there is a non-zero map 
\[
H^1\hspace{-0.1em}(C_{n,t}, \Q)\otimes _{\mathbb{Q}} H^1\hspace{-0.1em}(E_n, \Q)^{\otimes 2} \rightarrow H^1\hspace{-0.1em}(X_{n,t}, \Q)^{\otimes 3}
\]
for $n=7,9$. The Hodge conjecture then predicts a  correspondence between $C_{n,t} \times E_n^2$ and $X_{n,t}^3$; it would be interesting to find an explicit geometric construction of this correspondence.
\begin{rem}
We expect an analogue of Conjectures \ref{conj:HGM237}, \ref{conj:HGM239} to hold for the Shimura curve $\calX_{11}$. In this case the field $F_{11}$ turns out to be $\Q(\zeta_{11})^+(\sqrt{-33})$, and therefore the analogous elliptic curve $E_{11}$ should have CM by $\Q(\sqrt{-33})$.

Note that the universal family of polarized abelian varieties $A_t$ over $\calX_{11}$ might not be a (relative) Jacobian and consequently $H^1(C_{n,t}, \Q)$ should be replaced with $H^1(A_t, \Q)$.
\end{rem}
\newpage
\begin{appendix}
\section{Tables}\label{sec:tables}
\begin{table}[h!] \label{table:237}
  \begin{center}
  \caption{Putative CM points on $\calX_7$}
    \begin{tabular}[]{l|l|l}
    CM field $K$ & $\disc(K)$ & Field of definition\\
    \hline \hline
    \href{https://www.lmfdb.org/NumberField/8.0.7834003547041.1}{8.0.7834003547041.1} & $7^{4} \cdot 239^{4}$ & \href{https://www.lmfdb.org/NumberField/3.3.49.1}{3.3.49.1}\\
    \href{https://www.lmfdb.org/NumberField/8.0.9529871528401.1}{8.0.9529871528401.1} & $7^{4} \cdot 251^{4}$ & \href{https://www.lmfdb.org/NumberField/3.3.49.1}{3.3.49.1}\\
    \href{https://www.lmfdb.org/NumberField/8.0.49539201251281.1}{8.0.49539201251281.1} & $7^{4} \cdot 379^{4}$ & \href{https://www.lmfdb.org/NumberField/3.3.49.1}{3.3.49.1}\\
    \href{https://www.lmfdb.org/NumberField/8.0.139546236594961.1}{8.0.139546236594961.1} & $7^{4} \cdot 491^{4}$ & \href{https://www.lmfdb.org/NumberField/3.3.49.1}{3.3.49.1}\\
    \href{https://www.lmfdb.org/NumberField/8.0.214951987660081.3}{8.0.214951987660081.3} & $7^{4} \cdot 547^{4}$ & \href{https://www.lmfdb.org/NumberField/3.3.49.1}{3.3.49.1}\\
    \href{https://www.lmfdb.org/NumberField/8.0.234444145241761.14}{8.0.234444145241761.14} & $7^{4} \cdot 13^{4} \cdot 43^{4}$ & \href{https://www.lmfdb.org/NumberField/6.2.31213.1}{6.2.31213.1}\\
    \href{https://www.lmfdb.org/NumberField/8.0.1742605033519441.13}{8.0.1742605033519441.13} & $7^{4} \cdot 13^{4} \cdot 71^{4}$ & \href{https://www.lmfdb.org/NumberField/6.2.31213.1}{6.2.31213.1}\\
    \href{https://www.lmfdb.org/NumberField/8.0.2929564229451601.1}{8.0.2929564229451601.1} & $7^{4} \cdot 1051^{4}$ & \href{https://www.lmfdb.org/NumberField/3.3.49.1}{3.3.49.1}\\
    \href{https://www.lmfdb.org/NumberField/8.0.3254452511700481.17}{8.0.3254452511700481.17} & $7^{4} \cdot 13^{4} \cdot 83^{4}$ & \href{https://www.lmfdb.org/NumberField/6.2.31213.1}{6.2.31213.1}\\
    \href{https://www.lmfdb.org/NumberField/8.0.9197356446316801.6}{8.0.9197356446316801.6} & $7^{4} \cdot 1399^{4}$ & \href{https://www.lmfdb.org/NumberField/9.3.164590951.1}{9.3.164590951.1}\\
    \href{https://www.lmfdb.org/NumberField/8.0.12515552243039041.2}{8.0.12515552243039041.2} & $7^{4} \cdot 1511^{4}$ & \href{https://www.lmfdb.org/NumberField/9.3.177767639.1}{9.3.177767639.1}\\
    \href{https://www.lmfdb.org/NumberField/8.0.15077030712450721.1}{8.0.15077030712450721.1} & $7^{4} \cdot 1583^{4}$ & \href{https://www.lmfdb.org/NumberField/9.3.186238367.1}{9.3.186238367.1}\\
    \href{https://www.lmfdb.org/NumberField/8.0.17839408610934001.14}{8.0.17839408610934001.14} & $7^{4} \cdot 13^{4} \cdot 127^{4}$ & \href{https://www.lmfdb.org/NumberField/6.2.31213.1}{6.2.31213.1}\\
    \href{https://www.lmfdb.org/NumberField/8.0.23195375592678961.2}{8.0.23195375592678961.2} & $7^{4} \cdot 41^{4} \cdot 43^{4}$ & \href{https://www.lmfdb.org/NumberField/6.2.98441.1}{6.2.98441.1}\\
    \href{https://www.lmfdb.org/NumberField/8.0.25599104327834401.1}{8.0.25599104327834401.1} & $7^{4} \cdot 13^{4} \cdot 139^{4}$ & \href{https://www.lmfdb.org/NumberField/6.2.31213.1}{6.2.31213.1}\\
    \href{https://www.lmfdb.org/NumberField/8.0.33382684286885521.1}{8.0.33382684286885521.1} & $7^{4} \cdot 1931^{4}$ & \href{https://www.lmfdb.org/NumberField/9.3.227180219.1}{9.3.227180219.1}\\
    \href{https://www.lmfdb.org/NumberField/8.0.37426880324598961.1}{8.0.37426880324598961.1} & $7^{4} \cdot 1987^{4}$ & \href{https://www.lmfdb.org/NumberField/9.3.233768563.1}{9.3.233768563.1}\\
    \href{https://www.lmfdb.org/NumberField/8.0.50638501350060001.1}{8.0.50638501350060001.1} & $7^{4} \cdot 2143^{4}$ & \href{https://www.lmfdb.org/NumberField/9.3.252121807.1}{9.3.252121807.1}\\
    \href{https://www.lmfdb.org/NumberField/8.0.73350432097456801.1}{8.0.73350432097456801.1} & $7^{4} \cdot 2351^{4}$ & \href{https://www.lmfdb.org/NumberField/9.3.276592799.1}{9.3.276592799.1}\\
    \href{https://www.lmfdb.org/NumberField/8.0.135923763363916801.1}{8.0.135923763363916801.1} & $7^{4} \cdot 13^{4} \cdot 211^{4}$ & \href{https://www.lmfdb.org/NumberField/12.2.205567038859.1}{12.2.205567038859.1}\\
    \href{https://www.lmfdb.org/NumberField/8.0.169584057270610801.1}{8.0.169584057270610801.1} & $7^{4} \cdot 13^{4} \cdot 223^{4}$ & \href{https://www.lmfdb.org/NumberField/6.2.31213.1}{6.2.31213.1}\\
    \href{https://www.lmfdb.org/NumberField/8.0.172409416411292641.2}{8.0.172409416411292641.2} & $7^{4} \cdot 41^{4} \cdot 71^{4}$ & \href{https://www.lmfdb.org/NumberField/12.2.688034764151.1}{12.2.688034764151.1}\\
    \href{https://www.lmfdb.org/NumberField/8.0.200514022193502241.1}{8.0.200514022193502241.1} & $7^{4} \cdot 3023^{4}$ & \href{https://www.lmfdb.org/NumberField/15.3.2581408254768721.1}{15.3.2581408254768721.1}\\
    \href{https://www.lmfdb.org/NumberField/8.0.291354123972299521.1}{8.0.291354123972299521.1} & $7^{4} \cdot 3319^{4}$ & \href{https://www.lmfdb.org/NumberField/9.3.390477031.1}{9.3.390477031.1}\\
    \href{https://www.lmfdb.org/NumberField/8.0.295590648845235121.1}{8.0.295590648845235121.1} & $7^{4} \cdot 3331^{4}$ & \href{https://www.lmfdb.org/NumberField/9.3.391888819.1}{9.3.391888819.1}\\
    \href{https://www.lmfdb.org/NumberField/8.0.305655472188046561.1}{8.0.305655472188046561.1} & $7^{4} \cdot 3359^{4}$ & \href{https://www.lmfdb.org/NumberField/15.3.3187134619912369.1}{15.3.3187134619912369.1}\\
    \href{https://www.lmfdb.org/NumberField/8.0.321988200130081681.1}{8.0.321988200130081681.1} & $7^{4} \cdot 41^{4} \cdot 83^{4}$ & \href{https://www.lmfdb.org/NumberField/6.2.98441.1}{6.2.98441.1}\\
    \href{https://www.lmfdb.org/NumberField/8.0.512872428109556641.1}{8.0.512872428109556641.1} & $7^{4} \cdot 3823^{4}$ & \href{https://www.lmfdb.org/NumberField/9.3.449772127.1}{9.3.449772127.1}\\
    \href{https://www.lmfdb.org/NumberField/8.0.559455831332775601.1}{8.0.559455831332775601.1} & $7^{4} \cdot 3907^{4}$ & \href{https://www.lmfdb.org/NumberField/9.3.459654643.1}{9.3.459654643.1}\\
    \href{https://www.lmfdb.org/NumberField/8.0.609142738186488961.1}{8.0.609142738186488961.1} & $7^{4} \cdot 13^{4} \cdot 307^{4}$ & \href{https://www.lmfdb.org/NumberField/12.2.299095170283.1}{12.2.299095170283.1}\\
    \href{https://www.lmfdb.org/NumberField/8.0.616502043501426481.1}{8.0.616502043501426481.1} & $7^{4} \cdot 4003^{4}$ & \href{https://www.lmfdb.org/NumberField/9.3.470948947.1}{9.3.470948947.1}\\
    \href{https://www.lmfdb.org/NumberField/8.0.718369427913099361.1}{8.0.718369427913099361.1} & $7^{4} \cdot 4159^{4}$ & \href{https://www.lmfdb.org/NumberField/9.3.489302191.1}{9.3.489302191.1}\\
    \href{https://www.lmfdb.org/NumberField/8.0.726696250583406481.3}{8.0.726696250583406481.3} & $7^{4} \cdot 43^{4} \cdot 97^{4}$ & \href{https://www.lmfdb.org/NumberField/12.2.2332363542187.1}{12.2.2332363542187.1}\\
    \href{https://www.lmfdb.org/NumberField/8.0.798933157585310881.1}{8.0.798933157585310881.1} & $7^{4} \cdot 4271^{4}$ & \href{https://www.lmfdb.org/NumberField/9.3.502478879.1}{9.3.502478879.1}\\
    \href{https://www.lmfdb.org/NumberField/8.0.807949948577050321.1}{8.0.807949948577050321.1} & $7^{4} \cdot 4283^{4}$ & \href{https://www.lmfdb.org/NumberField/9.3.503890667.1}{9.3.503890667.1}\\
    \href{https://www.lmfdb.org/NumberField/8.0.851041543370918641.1}{8.0.851041543370918641.1} & $7^{4} \cdot 4339^{4}$ & \href{https://www.lmfdb.org/NumberField/15.3.5318139197378329.1}{15.3.5318139197378329.1}\\
    \href{https://www.lmfdb.org/NumberField/8.0.1004843693178804241.1}{8.0.1004843693178804241.1} & $7^{4} \cdot 4523^{4}$ & \href{https://www.lmfdb.org/NumberField/9.3.532126427.1}{9.3.532126427.1}\\
    \href{https://www.lmfdb.org/NumberField/8.0.1135151485651484161.1}{8.0.1135151485651484161.1} & $7^{4} \cdot 4663^{4}$ & \href{https://www.lmfdb.org/NumberField/15.3.6142020067423681.1}{15.3.6142020067423681.1}\\
    \end{tabular}
  \end{center}
\end{table}

\begin{table}[h!] \label{table:239}
  \begin{center}
  \caption{Putative CM points on $\calX_9$}
    \begin{tabular}[]{l|l|l}
      CM field $K$ & $\disc(K)$ & Field of definition\\
      \hline \hline
      \href{https://www.lmfdb.org/NumberField/8.0.166726039041.1}{8.0.166726039041.1} & $3^{8} \cdot 71^{4}$ & \href{https://www.lmfdb.org/NumberField/3.3.81.1}{3.3.81.1}\\
      \href{https://www.lmfdb.org/NumberField/8.0.10289217397761.2}{8.0.10289217397761.2} & $3^{8} \cdot 199^{4}$ & \href{https://www.lmfdb.org/NumberField/3.3.81.1}{3.3.81.1}\\
      \href{https://www.lmfdb.org/NumberField/8.0.135371386676241.1}{8.0.135371386676241.1} & $3^{8} \cdot 379^{4}$ & \href{https://www.lmfdb.org/NumberField/3.3.81.1}{3.3.81.1}\\
      \href{https://www.lmfdb.org/NumberField/8.0.419992928325441.1}{8.0.419992928325441.1} & $3^{8} \cdot 503^{4}$ & \href{https://www.lmfdb.org/NumberField/9.3.267314823.1}{9.3.267314823.1}\\
      \href{https://www.lmfdb.org/NumberField/8.0.490881644910801.1}{8.0.490881644910801.1} & $3^{8} \cdot 523^{4}$ & \href{https://www.lmfdb.org/NumberField/3.3.81.1}{3.3.81.1}\\
      \href{https://www.lmfdb.org/NumberField/8.0.1956806090111601.2}{8.0.1956806090111601.2} & $3^{8} \cdot 739^{4}$ & \href{https://www.lmfdb.org/NumberField/3.3.81.1}{3.3.81.1}\\
      \href{https://www.lmfdb.org/NumberField/8.0.5832395381554641.2}{8.0.5832395381554641.2} & $3^{8} \cdot 971^{4}$ & \href{https://www.lmfdb.org/NumberField/9.3.516029211.1}{9.3.516029211.1}\\
      \href{https://www.lmfdb.org/NumberField/8.0.45077252005749681.1}{8.0.45077252005749681.1} & $3^{8} \cdot 1619^{4}$ & \href{https://www.lmfdb.org/NumberField/15.3.9139423287309561.1}{15.3.9139423287309561.1}\\
      \href{https://www.lmfdb.org/NumberField/8.0.80401691902790241.1}{8.0.80401691902790241.1} & $3^{8} \cdot 1871^{4}$ & \href{https://www.lmfdb.org/NumberField/15.3.12205980432301041.1}{15.3.12205980432301041.1}\\
      \href{https://www.lmfdb.org/NumberField/8.0.86770641842707761.1}{8.0.86770641842707761.1} & $3^{8} \cdot 1907^{4}$ & \href{https://www.lmfdb.org/NumberField/15.3.12680211005112249.1}{15.3.12680211005112249.1}\\
      \href{https://www.lmfdb.org/NumberField/8.0.147910484188525041.1}{8.0.147910484188525041.1} & $3^{8} \cdot 2179^{4}$ & \href{https://www.lmfdb.org/NumberField/9.3.1158009939.1}{9.3.1158009939.1}\\
      \href{https://www.lmfdb.org/NumberField/8.0.179487701547492321.1}{8.0.179487701547492321.1} & $3^{8} \cdot 2287^{4}$ & \href{https://www.lmfdb.org/NumberField/9.3.1215405567.1}{9.3.1215405567.1}\\
      \href{https://www.lmfdb.org/NumberField/8.0.196377137395449201.1}{8.0.196377137395449201.1} & $3^{8} \cdot 2339^{4}$ & \href{https://www.lmfdb.org/NumberField/15.3.19075922001903321.1}{15.3.19075922001903321.1}\\
      \href{https://www.lmfdb.org/NumberField/8.0.235237153286975841.1}{8.0.235237153286975841.1} & $3^{8} \cdot 2447^{4}$ & \href{https://www.lmfdb.org/NumberField/21.3.335194620603462704888703.1}{21.3.335194620603462704888703.1}\\
      \href{https://www.lmfdb.org/NumberField/8.0.257521466109501441.1}{8.0.257521466109501441.1} & $3^{8} \cdot 2503^{4}$ & \href{https://www.lmfdb.org/NumberField/9.3.1330196823.1}{9.3.1330196823.1}\\
      \href{https://www.lmfdb.org/NumberField/8.0.295692078251129121.1}{8.0.295692078251129121.1} & $3^{8} \cdot 2591^{4}$ & \href{https://www.lmfdb.org/NumberField/21.3.397921468909806398973231.1}{21.3.397921468909806398973231.1}\\
      \href{https://www.lmfdb.org/NumberField/8.0.348162166355287761.1}{8.0.348162166355287761.1} & $3^{8} \cdot 2699^{4}$ & \href{https://www.lmfdb.org/NumberField/15.3.25399833134309001.1}{15.3.25399833134309001.1}\\
      \href{https://www.lmfdb.org/NumberField/8.0.358597178472085281.1}{8.0.358597178472085281.1} & $3^{8} \cdot 2719^{4}$ & \href{https://www.lmfdb.org/NumberField/9.3.1444988079.1}{9.3.1444988079.1}\\
      \href{https://www.lmfdb.org/NumberField/8.0.398115772466433921.1}{8.0.398115772466433921.1} & $3^{8} \cdot 2791^{4}$ & \href{https://www.lmfdb.org/NumberField/9.3.1483251831.1}{9.3.1483251831.1}\\
      \href{https://www.lmfdb.org/NumberField/8.0.428624978678095761.2}{8.0.428624978678095761.2} & $3^{8} \cdot 2843^{4}$ & \href{https://www.lmfdb.org/NumberField/15.3.28182454451958249.1}{15.3.28182454451958249.1}\\
    \end{tabular}
\end{center}
\end{table}
\FloatBarrier
\end{appendix}

\bibliographystyle{alphaurl}
\bibliography{references.bib}


\end{document}